\documentclass[a4paper, 11pt, reqno]{amsart}

\usepackage{amssymb, amsmath}
\usepackage[margin=2.5cm]{geometry}

\newtheorem{theorem}{Theorem}[section]
\newtheorem{lemma}[theorem]{Lemma}
\newtheorem{lem}[theorem]{Lemma}
\newtheorem{conj}[theorem]{Conjecture}
\newtheorem{prop}[theorem]{Proposition}
\newtheorem{claim}[theorem]{Claim}

\theoremstyle{definition}

\theoremstyle{remark}

\newcommand{\Pb}{\mathbb{P}}

\newcommand{\E}{\mathbb{E}}
\newcommand{\Var}{\text{Var}}
\newcommand{\Cov}{\text{Cov}}
\DeclareMathOperator{\ind}{ind}

\numberwithin{equation}{section}

\begin{document}

\title{Edge-statistics on large graphs}

\author{Noga Alon}
\address{Schools of Mathematics and Computer Science, Tel Aviv University, Tel Aviv 6997801, Israel and Department
of Mathematics, Princeton University, Princeton, NJ 08544, USA.
}
\email{nogaa@tau.ac.il}
\thanks{The research of Noga Alon is supported in part by  
ISF grant 281/17 and by GIF 
grant G-1347-304.6/2016.}

\author{Dan Hefetz}
\address{Department of Computer Science, Ariel University, Ariel 40700, Israel.}
\email{danhe@ariel.ac.il}
\thanks{The research of Dan Hefetz is supported by ISF grant 822/18.}

\author{Michael Krivelevich}
\address{Sackler School of Mathematical Sciences, Tel Aviv University, Tel Aviv 6997801, Israel.}
\email{krivelev@post.tau.ac.il }
\thanks{The research of Michael Krivelevich is supported in part by BSF grant 2014361 and ISF grant 1261/17.}

\author{Mykhaylo Tyomkyn}
\address{Sackler School of Mathematical Sciences, Tel Aviv University, Tel Aviv 6997801, Israel.}
\email{tyomkynm@post.tau.ac.il}
\thanks{The research of Mykhaylo Tyomkyn is supported in part by ERC Starting Grant 633509.}


\date{\today}

\begin{abstract}
The inducibility of a graph $H$ measures the maximum number of induced copies of $H$ a large graph $G$ can have. Generalizing this notion, we study how many induced subgraphs of fixed order $k$ and size $\ell$ a large graph $G$ on $n$ vertices can have. Clearly, this number is $\binom{n}{k}$ for every $n$, $k$ and $\ell \in \left \{0, \binom{k}{2} \right\}$. We conjecture that for every $n$, $k$ and $0 < \ell < \binom{k}{2}$ this number is at most $\left(1/e + o_k(1) \right) \binom{n}{k}$. If true, this would be tight for $\ell \in \{1, k-1\}$.  

In support of our `Edge-statistics conjecture' we prove that the corresponding density is bounded away from $1$ by an absolute constant. Furthermore, for various ranges of the values of $\ell$ we establish stronger bounds. In particular, we prove that for `almost all' pairs $(k, \ell)$ only a polynomially small fraction of the $k$-subsets of $V(G)$ have exactly $\ell$ edges, and prove an upper bound of $(1/2 + o_k(1))\binom{n}{k}$ for $\ell = 1$. 

Our proof methods involve probabilistic tools, such as anti-concentration results relying on fourth moment estimates and Brun's sieve, as well as graph-theoretic and combinatorial arguments such as Zykov's symmetrization, Sperner's theorem and various counting techniques.
\end{abstract}

\maketitle

\section{Introduction}

Let $k$ be a positive integer and let $G$ be a finite graph of order at least $k$. 
Let $A=A_{G,k}$ be chosen uniformly at random from all subsets of $V(G)$ of size $k$ and let $X_{G,k} = e(G[A])$. That is, $X_{G,k}$ is the random variable counting the number of edges of $G$ with both endpoints in $A$. Naturally, the above quantities can also be interpreted as densities rather than probabilities, and we shall frequently switch between these two perspectives. 

Given integers $n \geq k$ and $0 \leq \ell \leq \binom{k}{2}$, let $I(n, k, \ell) = \max\{\Pb(X_{G,k} = \ell) : |G| = n\}$, that is, $I(n, k, \ell)$ is the maximum density of induced subgraphs with $k$ vertices and $\ell$ edges, taken over all graphs of order $n$. A standard averaging argument shows that $I(n,k,\ell)$ is a monotone decreasing function of $n$. Consequently, we define $\ind(k,\ell) := \lim_{n\rightarrow \infty} I(n,k,\ell)$ to be the \emph{edge-inducibility} of $k$ and $\ell$. While this quantity is trivially $1$ for $\ell \in \left\{0, \binom{k}{2} \right\}$ (simply take $G$ to be a large empty or complete graph, respectively), it is natural to ask how large can $\ind(k, \ell)$ be for $0 < \ell < \binom{k}{2}$. 

This question is closely related to the problem of determining the inducibilities of fixed graphs, a concept which was introduced in 1975 by Pippenger and Golumbic~\cite{PG}. For a graph $H$, let $D_H(G)$ denote the number of induced subgraphs of $G$ that are isomorphic to $H$, and let $I_H(n) = \max\{D_H(G) : |G| = n\}$. Again, the sequence $\{I_H(n)/\binom{n}{|H|}\}_{n=|H|}^{\infty}$ is monotone decreasing and thus converges to a limit $\ind(H)$, the \emph{inducibility} of $H$. Recently there has been a surge of interest in this area (see, e.g.,~\cite{Baletal, HT, Yuster, KNV}).

Observe that both types of inducibility are invariant under taking complements, that is, $\ind(k, \ell) = \ind \left(k,\binom{k}{2} - \ell \right)$
and $\ind(H) = \ind(\overline{H})$. Note also that $\ind(H) \leq \ind(|H|, e(H))$. Moreover, if $|H| = k$ and $e(H) \in \left\{1, \binom{k}{2} - 1 \right\}$, then $\ind(H) = \ind(k, e(H))$, as $H$ is the unique (up to isomorphism) graph with $k$ vertices and $e(H)$ edges. 

Consider a random graph $G \sim G(n,p)$, where $p = \binom{k}{2}^{-1}$. A straightforward calculation shows that the expected value of the number of $k$-subsets of $V(G)$ which span precisely one edge is about $1/e$. This implies that $\ind(k,1) \geq 1/e + o_k(1)$ (as the $o_k(1)$ notation suggests, we will often think of $k$ as an asymptotic quantity and, in particular, we will assume $k$ to be sufficiently large wherever needed). In fact, we will see later that there are many constructions which achieve $1/e + o_k(1)$ as a lower bound for $\ind(k,1)$. Another construction, achieving the same asymptotic value for $\ell = k-1$ is the complete bipartite graph with the smaller part of size $n/k$, so that $\ind(k,k-1) \geq \ind(K_{1,k-1}) \geq 1/e + o_k(1)$. In fact, it is known~\cite{BS} that $\ind(K_{1,k-1}) = 1/e + o_k(1)$.   
Note that the $o_k(1)$ term is necessary. For example, counting cherries in $K_{n/2,n/2}$ shows that $\ind(3,2) = \ind(K_{1,2}) \geq 3/4$ (in fact, it follows from Goodman's Theorem that $\ind(3,1) = \ind(3,2) = 3/4$).
Motivated by the aforementioned constructions (as well as some additional data), we conjecture that the lower bound of $1/e$ is asymptotically tight.
\begin{conj}[The Edge-statistics Conjecture] \label{conj:stat}
For every $\varepsilon > 0$ there exists $k_0 = k_0(\varepsilon)$ such that for all integers $k > k_0$ and $0 < \ell < \binom{k}{2}$ we have $\ind(k,\ell) \leq 1/e + \varepsilon$.
\end{conj}

For graph-inducibilities we make an analogous conjecture, which would be implied by the Edge-statistics Conjecture. 
\begin{conj}[The Large Inducibility Conjecture] \label{conj:ind}
\[\limsup \left\{\ind(H) :  H \notin \{K_{|H|}, \overline{K}_{|H|}\} \right\} = 1/e.\]
\end{conj}
Our first theorem in this paper constitutes a first step towards proving Conjecture~\ref{conj:stat}. It asserts that $\ind(k,\ell)$ is bounded away from $1$ by an absolute constant for every $k$ and $0 < \ell < \binom{k}{2}$.
\begin{theorem} \label{thm:main}
There exists an $\varepsilon > 0$ such that for all positive integers $k$ and $\ell$ which satisfy $0 < \ell < \binom{k}{2}$ we have 
$$
\ind(k,\ell) < 1 - \varepsilon.
$$ 
\end{theorem}
\noindent
For clarity of presentation, we do not give explicit bounds on $\varepsilon$ and refer to Section~\ref{sec:remarks} for a discussion. 

Note that it is not hard to prove that for every positive integer $k$ we have $\ind(k, \ell) = 1$ if and only if $\ell \in \left\{0, \binom{k}{2} \right\}$. Indeed, if $0 < \ell < \binom{k}{2}$, then $\ind(k,\ell) < 1 - 4^{-k^2}$ is an easy consequence of Ramsey's Theorem and the aforementioned monotonicity of $I(n, k, \ell)$. With a bit more effort, this bound can be improved to $1 - k^{-2}$. On the other hand, we do not see a simple argument that would upper bound $\ind(k,\ell)$ away from $1$ by an absolute constant as in Theorem~\ref{thm:main}. Note also that the related problem of \emph{minimizing} graph-inducibilities has been extensively studied. In particular, Pippenger and Golumbic~\cite{PG} showed that the inducibility of any $k$-vertex graph is at least $(1 + o_k(1)) k!/k^k$. It follows that $\ind(H) > 0$ for every graph $H$ and thus $\ind(k,\ell) > 0$ for every $k$ and $\ell$. We refer the reader to Section~\ref{sec:remarks} for further discussion.

For various ranges of values of $\ell$ (viewed as a function of $k$) we establish much better upper bounds than the one stated in Theorem~\ref{thm:main}. First, for every $\ell$ satisfying $\min \left\{\ell, \binom{k}{2} - \ell \right\} = \omega(k)$, we prove an upper bound of $1/2$.
\begin{prop} \label{prop:largelweak}
For every $\varepsilon > 0$ there exist $C(\varepsilon) > 0$ and $k_0(\varepsilon) > 0$ such that the following holds. Let $k$ and $\ell$ be integers satisfying $k \geq k_0$ and $C k \leq \ell \leq \binom{k}{2} - C k$. Then  
$$
\ind(k,\ell) \leq \frac{1}{2} + \varepsilon.
$$
\end{prop}

Next, we prove Conjecture~\ref{conj:stat} `almost everywhere'. In fact, we prove a much stronger statement, namely that for every $\ell$ satisfying $\min \left\{\ell, \binom{k}{2} - \ell \right\} = \Omega \left(k^2 \right)$ the quantity $\ind(k,\ell)$ is actually polynomially small in $k$ -- the right asymptotic behavior as can be seen by considering the random graph $G(n,\ell/\binom{k}{2})$, which gives $\ind(k,\ell)=\Omega(k^{-1})$. 

\begin{theorem}\label{thm:poly}
For every positive integers $k$ and $\ell$ such that $\min\{\ell, k^2/2-\ell\} = \Omega(k^2)$ we have $\ind(k,\ell) = O \left(k^{-0.1} \right)$.
\end{theorem}
\noindent

Lastly, we consider the case when $\ell$ is fixed (i.e., does not depend on $k$). Here we prove an upper bound of $3/4$. In the interesting sub-case $\ell=1$, which corresponds to the inducibility of the one-edge graph (equivalently, of $K_k^-$, the complete graph with one edge removed) we prove a yet better bound of $1/2$.
\begin{theorem} \label{thm:fixed}
For every fixed positive integer $\ell$ we have 
$$
\ind(k,\ell) \leq \frac{3}{4} + o_k(1).
$$
Moreover, for $\ell=1$ we have
$$
\ind(k,1) \leq \frac{1}{2} + o_k(1).
$$ 
\end{theorem}

\noindent
Our results are summarized in the following table. For various ranges of $\ell \leq k^2/4$, it states the best known upper bound on $\ind(k, \ell)$. Note that for $\ell \geq k^2/4$ the table can be extended symmetrically.

\medskip

\begin{center}
  \begin{tabular}{| c | c | c | c | c | c |}
    \hline
    $\ell = \ell(k)$ & $1$ & const. & $\left[\omega(1), O(k)\right]$ & $\left[\omega(k), o(k^2)\right]$ & $\left[\Omega(k^2), {k^2}/{4}\right]$ \\ \hline
    $ind(k,\ell) \leq$ & ${1}/{2}$ & ${3}/{4}$ & $1-\varepsilon$ & ${1}/{2}$ & $O(k^{-0.1})$\\  
    \hline
  \end{tabular}
\end{center} 

\medskip

\subsection{Notation}

\noindent Throughout this paper, $\log$ stands for the natural logarithm, unless explicitly stated otherwise. For positive integers $n \geq k$ we denote by $(n)_k$ the \emph{falling factorial} $\prod_{i=0}^{k-1} (n-i)$. The \emph{symmetric difference} of two sets $A$ and $B$, denoted by $A \triangle B$, is $(A \setminus B) \cup (B \setminus A)$. 

Our graph-theoretic notation is standard and follows that of~\cite{Byellow}. In particular, we use the following. For a graph $G$, let $V(G)$ and $E(G)$ denote its sets of vertices and edges respectively, and let $|G| = |V(G)|$ and $e(G) = |E(G)|$. The \emph{complement} of $G$, denoted by $\overline{G}$, is the graph with vertex set $V(G)$ and edge set $\binom{V(G)}{2} \setminus E(G)$. For a set $S \subseteq V(G)$, let $G[S]$ denote the graph with vertex set $S$ and edge set $\{uv \in E(G) : u, v \in S\}$. For disjoint sets $S, T \subseteq V(G)$, let $G[S,T]$ denote the bipartite graph with parts $S$ and $T$ and edge-set $\{uv \in E(G) : u \in S, v \in T\}$. For a set $S \subseteq V(G)$ and a vertex $v \in V(G)$, let $N_G(v, S) = \{u \in S : uv \in E(G)\}$ denote the \emph{neighbourhood} of $v$ in $S$ and let $d_G(v, S) = |N_G(v, S)|$ denote the \emph{degree} of $v$ into $S$. We abbreviate $N_G(v, V(G))$ under $N_G(v)$ and $d_G(v, V(G))$ under $d_G(v)$; we refer to the former
  as the \emph{neighbourhood} of $v$ in $G$ and to the latter as the \emph{degree} of $v$ in $G$. The \emph{maximum degree} of a graph $G$ is $\Delta(G) = \max \{d_G(u) : u \in V(G)\}$. Often, when there is no risk of confusion, we omit the subscript $G$ from the notation above. 

The rest of this paper is organized as follows. In Section~\ref{sec:prelim} we establish a number of facts and lemmas which will be useful later on when we will upper bound $\ind(k,\ell)$; we then prove Proposition~\ref{prop:largelweak}. In Sections~\ref{sec::main} and \ref{sec:poly} we prove Theorems~\ref{thm:main} and~\ref{thm:poly} respectively. Moving on to the fixed $\ell$ regime, in Section~\ref{sec:fixed} we establish some additional tools and prove Theorem~\ref{thm:fixed}. In Section~\ref{sec:remarks} we conclude the paper with several remarks and open problems.


\section{Preliminaries} \label{sec:prelim}
\noindent
Recall that $A=A_{G,k}$ is the set chosen uniformly at random from all subsets of $V(G)$ of size $k$ and $X_{G,k} = e(G[A])$. To simplify notation we abbreviate $X_{G,k}$ to $X$ whenever there is no risk of confusion. Our first lemma provides a useful global-local criterion for handling edge-inducibilities.

\begin{lemma} \label{lem:zykov}
Let $k$ and $\ell$ be positive integers satisfying $0 < \ell < \binom{k}{2}$ and let $a = \ind(k, \ell)$. Let $n$ be a sufficiently large integer and let $G$ be a graph on $n$ vertices which attains $I(n,k,\ell)$. Then, for every vertex $v \in V(G)$, we have $\Pb(X = \ell \mid v \in A) = a + o_n(1)$.
\end{lemma}

\begin{proof}
The main idea of the proof is the same as in the proof of Lemma 2.4 from~\cite{HT}. 
Double-counting yields 
\begin{eqnarray*}
a+o_n(1) &=& \Pb(X = \ell) = \frac{1}{k} \cdot \sum_{v \in V(G)} \Pb(X = \ell, v \in A) = \frac{1}{k} \cdot \sum_{v \in V(G)} \Pb(X = \ell \mid v \in A) \cdot \Pb(v \in A) \\ 
&=& \frac{1}{k} \cdot \frac{k}{n} \cdot \sum_{v \in V(G)} \Pb(X = \ell \mid v \in A) = \frac{1}{n} \cdot \sum_{v \in V(G)} \Pb(X = \ell \mid v \in A).
\end{eqnarray*}

Let $v^+$ and $v^-$ be the vertices with the largest and the smallest value of $\Pb(X=\ell \mid v\in A)$, respectively. Let the graph $G'$ be obtained from $G$ by Zykov's symmetrization~\cite{Zy}, i.e., remove $v^-$ and add a twin copy of $v^+$ instead (say, the two copies of $v^+$ are not connected by an edge in $G'$). Then
\begin{align*}
\Pb(X = \ell) &\geq \Pb(X_{G',k} = \ell) \geq \Pb(X = \ell) - \Pb(X = \ell, v^- \in A) + \Pb(X = \ell, v^+ \in A) \\
&- \Pb(X = \ell, v^- \in A, v^+ \in A) \\
&= \Pb(X = \ell) - \frac{k}{n} \cdot \Pb \left(X = \ell \mid v^- \in A \right) 
+ \frac{k}{n} \cdot \Pb \left(X = \ell \mid v^+ \in A \right) \\
&- \frac{k(k-1)}{n(n-1)} \cdot \Pb \left(X = \ell \mid v^+, v^- \in A \right)\\
&= \Pb(X = \ell) + \frac{k}{n} \cdot \left(\Pb(X = \ell \mid v^+ \in A) - \Pb(X = \ell \mid v^- \in A)\right) - O(n^{-2}),
\end{align*}
where the first inequality follows from our assumption that $G$ maximizes $\Pb(X_{G,k} = \ell)$.

Therefore
$$
\Pb(X = \ell \mid v^+ \in A) - \Pb(X = \ell \mid v^- \in A) = O(n^{-1}).
$$
Hence, 
$$
\Pb(X = \ell \mid v^+ \in A) \leq a + O(n^{-1})  \ \ \text{and} \ \  \Pb(X = \ell \mid v^- \in A) \geq a - O(n^{-1}).
$$
Since $a > 0$ (as remarked in the introduction), this concludes the proof of the lemma.
\end{proof}

\noindent
For two vertices $v,w \in V(G)$ we now consider a subset of $V(G) \setminus \{v,w\}$ of size $k-1$, chosen uniformly at random among all such subsets. Our next lemma shows that, assuming that $\ind(k,\ell)$ is large, with a significant probability $v$ and $w$ will have the same degree into this set. 

\begin{lemma} \label{lem:sameDegree} 
Let $k$ and $\ell$ be positive integers satisfying $0 < \ell < \binom{k}{2}$, let $a = \ind(k, \ell)$, and suppose that $a>1/2$.
Let $n$ be a sufficiently large integer and let $G$ be a graph on $n$ vertices which attains $I(n,k,\ell)$. Then for any two vertices $v,w \in V(G)$ we have 
$$
\Pb\left(e_G(v, A_{G\setminus \{v,w\},k-1}) = e_G(w, A_{G\setminus \{v,w\},k-1})\right) > 2a - 1 - o_n(1).
$$
\end{lemma}

\begin{proof}
By Lemma~\ref{lem:zykov} we have
\begin{equation} \label{eq:vNOTw}
\Pb(X = \ell \mid v \in A, w \notin A) = \Pb(X = \ell \mid v \in A) + o(1)  = a + o_n(1).
\end{equation}
and, symmetrically,
\begin{equation} \label{eq:wNOTv}
\Pb(X = \ell \mid w \in A, v \notin A) = a + o_n(1).
\end{equation}
Setting $G'' = G \setminus \{v, w\}$ and $B = A_{G\setminus \{v,w\},k-1}$, identities~\eqref{eq:vNOTw} and~\eqref{eq:wNOTv} imply that
$$
\Pb(X_{G'',k-1} + e_G(B,v) = \ell)  = a + o_n(1) \ \ \text{and} \ \ \Pb(X_{G'',k-1} + e_G(B,w) = \ell)  = a + o_n(1).
$$
Let $L_v$ (respectively, $L_w$) denote the event $X_{G'',k-1} + e_G(B,v) = \ell$ (respectively, $X_{G'',k-1} + e_G(B,w) = \ell$). Then
$$
\Pb(e_G(B, v) = e_G(B, w)) \geq \Pb(L_v \cap L_w) = \Pb(L_v) + \Pb(L_w) - \Pb(L_v \cup L_w) > 2a - 1 - o_n(1).
$$
\end{proof}
\begin{lemma} \label{lem:symmdiff}
For every $1/2 < a < 1$ there exists $C = C(a) > 0$ for which the following holds. Suppose that $k$ and $\ell$ are positive integers with $0 < \ell < \binom{k}{2}$ such that $a = \ind(k, \ell)$. Suppose that $n$ is sufficiently large and that $G$ is a graph on $n$ vertices which attains $I(n,k,\ell)$. Then for any two vertices $v,w \in V(G)$ we have 
\begin{align}~\label{eq:Lem23}
|N_G(v) \triangle N_G(w)| < Cn/k.
\end{align}
Moreover, as $a \rightarrow 1$, inequality~\eqref{eq:Lem23} holds with $C\rightarrow 0$.
\end{lemma}

Note that the second part of the statement of Lemma~\ref{lem:symmdiff} (i.e., the one referring to $a \rightarrow 1$) is only needed to prove Theorem~\ref{thm:main} by contradiction, and is otherwise vacuous, as it contradicts Theorem~\ref{thm:main}. 

\begin{proof}
We can assume $k$ to be larger than any given absolute constant, since, in the first statement, the small values of $k$ can be accommodated for by adjusting $C$, and the second statement implicitly assumes $k \rightarrow \infty$, since $a = \ind(k,\ell) < 1$ for any given $0 < \ell < \binom{k}{2}$, and takes only finitely many values when $k$ is bounded.

Fix two arbitrary vertices $v,w \in V(G)$ and let 
$$
P = N_G(v) \setminus (w\cup N_G(w)), \  R = N_G(w) \setminus (v\cup N_G(v)) \ \  \text{and} \ \ Q = V(G) \setminus (\{v,w\} \cup P \cup R),
$$ 
Let $B = A_{G \setminus \{v,w\}, k-1}$. 
It then follows by Lemma~\ref{lem:sameDegree} that
\begin{equation}\label{eq:BcapPBcapR}
\Pb(|B\cap P| = |B\cap R|) > 2a-1-o_n(1).
\end{equation}
Suppose that 
$|P \cup R| \geq cn/k$ for some absolute constant $c > 0$. Then 
\begin{equation} \label{eq:BPR}
\Pb(B \cap (P \cup R) \neq \emptyset) \geq 1 - (1 + o(1))(1 - c/k)^{k-1} = 1 - e^{-c} + o_k(1) = \Omega_k(1).
\end{equation}
In particular, if $|P\cup R| =\omega(n/k)$, then $\Pb(B \cap (P \cup R) \neq \emptyset) = 1 - o_k(1)$.
Therefore, using~\eqref{eq:BcapPBcapR} we obtain 
\begin{align}\label{eq:capnonzero}
\Pb \left(|B \cap P| = |B \cap R| \, \Big\rvert \, B \cap (P \cup R) \neq \emptyset \right) 
&=\frac{\Pb(|B \cap P| = |B \cap R| \wedge B \cap (P \cup R) \neq \emptyset)}{\Pb(B\cap (P\cup R)\neq \emptyset)}\nonumber \\
&= \frac{2a-1-o(1)}{1-o_k(1)}=2a-1-o_k(1).
\end{align}

\noindent
Similarly, if $a=1-o_k(1)$ and $|P\cup R|=\Omega(n/k)$, then~\eqref{eq:BcapPBcapR} implies that 
$\Pb (|B \cap P| = |B \cap R|)=1-o_k(1)$. Therefore
\begin{align}\label{eq:capnonzero2}
\Pb \left(|B \cap P| = |B \cap R| \, \Big\rvert \, B \cap (P \cup R) \neq \emptyset \right) 
&=\frac{\Pb(|B \cap P| = |B \cap R| \wedge B \cap (P \cup R) \neq \emptyset)}{\Pb(B\cap (P\cup R)\neq \emptyset)}\nonumber \\
&{=}\frac{\Pb(B\cap (P\cup R)\neq \emptyset)-o_k(1)}{\Pb(B\cap (P\cup R)\neq \emptyset)}\nonumber \\
&\stackrel{\eqref{eq:BPR}}{=}1-o_k(1).
\end{align}
Note that the above argument does not make any use of the graph structure of $G$. Indeed, the situation at hand can be viewed as an urn model, in which we have a large urn filled with $|P|$ pink balls and $|R|$ red balls, and we draw a fixed but otherwise arbitrary number $1 \leq s \leq k-1$ of balls from the urn  uniformly at random, without replacement. We would like to upper bound the probability of drawing equally many pink balls and red balls. To this end, we first prove the following auxiliary claim.
\begin{claim}\label{cl:stirling}
For every integer $1 \leq t \leq \lfloor (k-1)/2 \rfloor$ we have 
$$
\Pb(|B \cap P| = |B \cap R| = t) \leq \frac{3}{4} \cdot \Pb\left(|B \cap (P \cup R)| = 2t \right).
$$
Moreover, for every $\varepsilon > 0$ there exists $t_0 = t_0(\varepsilon)$ such that for every $t \geq t_0$ we have
$$
\Pb(|B \cap P| = |B \cap R| = t) \leq \varepsilon \cdot \Pb\left(|B \cap (P \cup R)| = 2t \right).
$$
\end{claim}

\begin{proof}
Fix some $1 \leq t \leq \lfloor (k-1)/2 \rfloor$. By the log-concavity of the binomial coefficients we have 
$$
\binom{|P|}{t}\binom{|R|}{t} \leq \binom{\frac{|P|+|R|}{2}}{t}^2.
$$
Moreover, using the fact that $t$ is much smaller than $|P|+|R|$ (as $t<k$ and $|P|+|R| \geq cn/k$ by assumption, and thus can be assumed to be sufficiently large), straightforward calculations show that
$$
\binom{\frac{|P|+|R|}{2}}{t}^2 \leq \frac{3}{4} \binom{|P|+|R|}{2t}.
$$  
We conclude that
\begin{eqnarray*}
\Pb(|B\cap P|=|B\cap R|=t) &=& \frac{\binom{|P|}{t} \binom{|R|}{t} \binom{n - 2 - |P| - |R|}{k-1-2t}}{\binom{n-2}{k-1}} \leq \frac{3}{4} \cdot \frac{\binom{|P|+|R|}{2t} \binom{n - 2 - |P| - |R|}{k-1-2t}}{\binom{n-2}{k-1}} \\ 
&=& \frac{3}{4} \cdot \Pb\left(|B \cap (P \cup R)| = 2t \right).
\end{eqnarray*}
The second statement can be proved analogously; we omit the details.
\end{proof}

\noindent
Coming back to the proof of Lemma~\ref{lem:symmdiff}, using Claim~\ref{cl:stirling} with $a = 1-o_k(1)$ we have
\begin{align*}
&\Pb(|B\cap P| = |B\cap R| > 0) = \sum_{t=1}^{\lfloor(k-1)/2\rfloor} \Pb(|B\cap P| = |B\cap R|=t) \\ 
&\leq \frac{3}{4} \cdot \sum_{t=1}^{\lfloor(k-1)/2\rfloor} \Pb(|B\cap (P\cup R)|=2t) \leq \frac{3}{4} \cdot \Pb(|B\cap (P\cup R)| > 0).
\end{align*}
Therefore
$$
\Pb \left(|B\cap P| = |B\cap R| \, \Big\rvert \, B \cap (P\cup R) \neq \emptyset \right) \leq \frac{3}{4}, 
$$
contrary to~\eqref{eq:capnonzero2}. Similarly, for any $a > 1/2$, if $|P\cup R|=\omega(n/k)$, we obtain that for some $t = t(a)$ we have 
$$
\Pb \left(|B\cap P| = |B\cap R| \, \Big\rvert \, |B \cap (P\cup R)| > t \right) < \frac{1}{2} \cdot (2a-1), 
$$
contrary to~\eqref{eq:capnonzero}.
This concludes the proof of Lemma~\ref{lem:symmdiff}.
\end{proof}
Under closer inspection, Lemma~\ref{lem:symmdiff} has the following immediate consequence.
\begin{lemma} \label{lem:maxdeg}
For every $1/2 < a < 1$ there exists $C = C(a) > 0$ for which the following holds. Suppose that $k$ and $\ell$ are positive integers with $0 < \ell < \binom{k}{2}$ such that $a = \ind(k, \ell)$. Suppose that $n$ is sufficiently large and that $G$ is a graph on $n$ vertices which attains $I(n,k,\ell)$. Suppose that $e(G) \leq \binom{n}{2}/2$. Then  $\Delta(G) < Cn/k$.
Moreover, as $a\rightarrow 1$, the above holds with $C\rightarrow 0$. 
%
\end{lemma}

\begin{proof}
We prove the first statement -- the second can be proven analogously. It follows from Lemma~\ref{lem:symmdiff} that $|d(v) - d(w)| = O(n/k)$ holds for any $v, w \in V(G)$. Let $v$ be a vertex of minimum degree in $G$. Put $U := N_G(v)$ and $W := N_{\overline{G}}(v)$. Observe that $|U| \leq |W|$ since we assumed that $d(v) = \delta(G)$ and $e(G) \leq \binom{n}{2}/2$. Suppose for a contradiction that $|U| =\omega(n/k)$.
We double-count the edges of the bipartite graph $G[U,W]$. Applying Lemma~\ref{lem:symmdiff} to $v$ and $u$ for every $u \in U$, we derive that $d_{G[U,W]}(u) = O(n/k)$. In particular
$$
e(G[U,W]) = O \left(|U| \cdot \frac{n}{k} \right) = O \left(|W| \cdot \frac{n}{k} \right).
$$ 
On the other hand, applying Lemma~\ref{lem:symmdiff} to $v$ and $w$ for every $w \in W$, yields 
$$
d_{G[U,W]}(w) \geq |U| - O \left(\frac{n}{k} \right) \geq \omega \left(\frac{n}{k}\right) - O \left(\frac{n}{k} \right) = \omega \left(\frac{n}{k}\right).
$$ 
Therefore 
$$
|W| \cdot \omega \left(\frac{n}{k}\right) \leq e(G[U,W]) = O \left(|W| \cdot \frac{n}{k} \right),
$$ 
which is clearly a contradiction, and thus $\delta(G) = O(n/k)$. Since, moreover, $|d(v) - d(w)| = O(n/k)$ holds for any $v, w \in V(G)$ by Lemma~\ref{lem:symmdiff}, we conclude that $\Delta(G) = O(n/k)$ as claimed.  
\end{proof} 

From now on, we denote $e(G)$ by $m$. For every $e \in E(G)$, let $X_e$ be the indicator random variable for the event $e \in E(G[A])$, that is, $X_e = 1$ if both endpoints of $e$ are in $A$ and $X_e = 0$ otherwise. Observe that $X = \sum_{e \in E(G)} X_e$. Putting $\mu = \E(X)$ we have 
\begin{equation} \label{eq::expectation}
\mu = \sum_{e \in E(G)} \E(X_e) = m \cdot \frac{k(k-1)}{n(n-1)} = m \cdot \frac{(k)_2}{n^2} \cdot (1 + O(1/n)).
\end{equation}
\noindent
This has the following consequence. 
\begin{lem} \label{lem:avgdeg}
Let $k$ and $\ell$ be positive integers satisfying $0 < \ell < \binom{k}{2}$ and let $a = \ind(k, \ell)$. Let $n$ be a sufficiently large integer and let $G$ be a graph with $n$ vertices and $m$ edges which attains $I(n,k,\ell)$. Then
$$
m \geq (1 - o_k(1)) a \ell\frac{n^2}{k^2}.
$$ 
In particular, if $a = 1-o_k(1)$, then 
$$
m \geq (1 - o_k(1)) \frac{n^2}{k^2}.
$$
\end{lem}

\begin{proof}
Suppose for a contradiction that $m < (1 - \varepsilon) a \ell \left(n^2/k^2 \right)$ for some constant $\varepsilon > 0$. It then follows by~\eqref{eq::expectation} that $\mu < (1 - \varepsilon/2) a \ell$. On the other hand, since $X \geq 0$ and by the choice of $G$ we have $\mu \geq \ell \cdot \Pb(X = \ell) = (1 - o(1)) a \ell$, a contradiction.
\end{proof}

\noindent
Combining the above facts, we can immediately prove Proposition~\ref{prop:largelweak}.
\begin{proof}[Proof of Proposition~\ref{prop:largelweak}.]
Suppose for a contradiction that there exist $\varepsilon > 0$ and an integer $\ell = \ell(k)$ such that $Ck\leq \ell\leq \binom{k}{2} - Ck$ for some large $C>0$, and $\ind(k,\ell) = a > 1/2 + \varepsilon$. Let $G$ be a graph attaining $I(n,k,\ell)$, where $n$ is sufficiently large. By symmetry we may assume that $e(G) \leq \binom{n}{2}/2$. Then, on the one hand, by Lemma~\ref{lem:maxdeg} we have $\Delta(G) = O(n/k)$ entailing $e(G) = O(n^2/k)$. On the other hand, Lemma~\ref{lem:avgdeg} implies that 
$$
e(G) \geq (1/2 - o(1)) \ell \cdot \frac{n^2}{k^2} \geq \frac{C}{3}\cdot \frac{n^2}{k},
$$
which is a contradiction for large enough $C$. We conclude that $a \leq 1/2 + o(1)$ as claimed.
\end{proof}

\section{Proof of Theorem~\ref{thm:main}} \label{sec::main}

Our goal in this section is to prove Theorem~\ref{thm:main}. Since, as remarked in the introduction, $\ind(k,\ell)$ is never identically $1$ (assuming $0 < \ell < \binom{k}{2}$), it suffices to prove the theorem for large values of $k$. Hence, suppose now, for a contradiction, that the assertion of Theorem~\ref{thm:main} is false for arbitrarily large values of $k$. That is, for every $\varepsilon > 0$ there exist arbitrarily large values of $k$ such that for some $0 < \ell < \binom{k}{2}$ there will be arbitrarily large values of $n$ and graphs $G$ on $n$ vertices for which $\Pb(X_{G,k} = \ell) > 1 - \varepsilon$. We may assume that $G$ maximizes $\Pb(X_{G,k} = \ell)$ over all $n$-vertex graphs and, by symmetry, that $e(G) \leq \binom{n}{2}/2$. 

We would like to calculate the variance of $X_{G,k}$. Using Lemma~\ref{lem:maxdeg} and our assumption that $k$ is sufficiently large we obtain
\begin{align*}
\E(X^2) &= \E \left(\left(\sum_{e \in E(G)} X_e \right)^2 \right) = \mu + \sum_{\stackrel{(e,f) \in E(G)^2}{e \cap f = \emptyset}} \E(X_e X_f) + \sum_{\stackrel{(e,f) \in E(G)^2}{|e \cap f| = 1}} \E(X_e X_f)\\
&= \mu + m^2 \cdot \frac{(k)_4}{n^4} (1 + O_k(1/n)) + S \cdot \frac{(k)_3}{n^3} (1 + O_k(1/n)),
\end{align*}
where $S = \sum_{v \in V(G)} d(v)^2$. Therefore,
\begin{align*}
\Var(X) &= \E(X^2) - \mu^2 = \left(m \cdot \frac{(k)_2}{n^2} - m^2 \cdot \frac{(k)_2 \cdot (k)_2}{n^4} + m^2 \cdot \frac{(k)_4}{n^4} + S \cdot \frac{(k)_3}{n^3}\right) (1 + O_k(1/n))\\
&= \left(m \cdot \frac{(k)_2}{n^2} + S \cdot \frac{(k)_3}{n^3} + m^2 \cdot \frac{k(k-1)}{n^4} \cdot(k^2 - 5k + 6 - k^2 + k) \right) (1 + O_k(1/n))\\
&= \left(m \cdot \frac{k^2}{n^2} + S \cdot \frac{k^3}{n^3} - m^2 \cdot \frac{4 k^3}{n^4}\right) (1 + o(1)). 
\end{align*}

Since, by Lemma~\ref{lem:maxdeg} we have $m = o(n^2/k)$, and, consequently, $4 m^2 k^3/n^4 = o(m k^2/n^2)$, we can write 

\begin{equation}\label{eq:VarX}
\Var(X) = (1+o(1)) \left(m \cdot \frac{k^2}{n^2} + S \cdot \frac{k^3}{n^3}\right).
\end{equation}
 

We shall need the following anti-concentration inequality from~\cite{AGK}, involving the fourth moment of a random variable. 

\begin{lemma}[\cite{AGK}: Lemma 3.2(i)]\label{lem:agk}
Let $Y$ be a real random variable and suppose that its first, second and fourth moments satisfy $\mathbb{E}(Y) = 0$, $\mathbb{E} \left(Y^2 \right) = \sigma^2 > 0$ and $\mathbb{E} \left(Y^4 \right) \leq b \sigma^4$ for some constant $b > 0$. Then 
$$
\Pb(Y > 0) \geq \frac{1}{2^{4/3}b} \ \ \ \text{and} \ \ \ \Pb(Y < 0) \geq \frac{1}{2^{4/3}b}. 
$$
\end{lemma} 

The aim in the rest of this section is to prove the following lemma, which directly implies Theorem~\ref{thm:main}. 

\begin{lemma} \label{lem:truth}
There exists a constant $b > 0$ such that
\begin{equation} \label{eq:keylemma}
\E \left( \left(X_{G,k} - \mu \right)^4 \right) \leq b \cdot \Var(X_{G,k})^2.
\end{equation}
\end{lemma}

Having established this, we conclude the proof Theorem~\ref{thm:main} as follows. By Lemma~\ref{lem:truth} the random variable $Y := X - \mu$ satisfies the assumptions of Lemma~\ref{lem:agk}. Since the event $\{X=\ell\}$ is disjoint either from the event $\{Y>0\}$ or from the event $\{Y<0\}$, using Lemma~\ref{lem:agk} we obtain $\Pb(X=\ell) \leq 1 - \frac{1}{2^{4/3}b} < 1 - \varepsilon$, contrary to our assumption that $\Pb(X = \ell) > 1 - \varepsilon$. 

\begin{proof}[Proof of Lemma~\ref{lem:truth}]
First let us expand the right hand side of~\eqref{eq:keylemma}.
\begin{align}\label{eq:Varsq}
\Var(X)^2 \stackrel{\eqref{eq:VarX}}{=} (1+o(1)) \left(m \cdot \frac{k^2}{n^2} + S \cdot \frac{k^3}{n^3}\right)^2 = \Theta \left(m^2 \cdot \frac{k^4}{n^4} + m S \cdot  \frac{k^5}{n^5} + S^2 \cdot \frac{k^6}{n^6}\right).
\end{align}

\noindent
Now, putting $p := {(k)_2}/{(n)_2}$, the left hand side of~\eqref{eq:keylemma} can be written as
\begin{equation}\label{eq:X4expanded}
\E \left( \left(X - \mu \right)^4 \right) = \E \left(\left(\sum_{e\in E(G)} \left(X_e - p \right) \right)^4\right) = \sum_{(e,f,g,h) \in E(G)^4} \Cov(e,f,g,h),
\end{equation}
where, by abuse of notation,  
\begin{equation}\label{eq:Cov}
\Cov(e,f,g,h) := \E \left((X_e-p)(X_f-p)(X_g-p)(X_h-p) \right).
\end{equation}

The following simple technical claim will play a crucial role in our proof of Lemma~\ref{lem:truth}. 
\begin{claim} \label{obs:cov} 
For every $4$-tuple $(e,f,g,h) \in E(G)^4$ we have
\[\Cov(e,f,g,h) = O(\E(X_e X_f X_g X_h)).\]
\end{claim}

\begin{proof}
Fix some $4$-tuple $(e,f,g,h) \in E(G)^4$. Let $B$ be some subset of $\{e,f,g,h\}$. Let $H_B$ denote the graph spanned by the edges in $B$ and let $t = |V(H_B)|$. Let $H'_B$ be a graph obtained from $H_B$ by deleting one of its edges (together with any of its endpoints if its degree in $H$ is $1$), adding two new vertices, and connecting them by an edge. Let $t'$ denote the number of vertices of $H'_B$; clearly $t \leq t' \leq k$. Since $k \ll n$, a straightforward calculation show that
\begin{eqnarray*} 
\E \left(p \cdot \prod_{a \in E(H_B) \cap E(H'_B)} X_a \right) &=& \E \left(\prod_{a \in E(H'_B)} X_a \right) = O \left(\frac{k^{t'}}{n^{t'}} \right) \\ 
&=& O \left(\frac{k^t}{n^t} \right) = O \left(\E \left(\prod_{a \in E(H_B)} X_a \right) \right).
\end{eqnarray*}
Therefore, for every $B \subseteq \{e,f,g,h\}$ we have 
$$
p^{4 - |B|} \cdot \E \left(\prod_{a \in B} X_a \right) = O(\E(X_e X_f X_g X_h)).
$$
We conclude that 
$$
\Cov(e,f,g,h) = O(\E(X_e X_f X_g X_h)),
$$
as claimed.
\end{proof}

We now distinguish between several cases, depending on how the four edges $e,f,g,h$ are arranged. Suppose first that $e,f,g,h$ are pairwise distinct. Let $H$ be the four-edge graph spanned by $e, f, g$ and $h$. Note that $H$ is a subgraph of $G$, but \emph{not} necessarily an induced one. The options for $H$ are

\begin{itemize}
\item[(i)] $H = 4 K_2$;
\item[(ii)] $H = 2 K_2 + K_{1,2}$;
\item[(iii)] $H = 2 K_{1,2}$;
\item[(iv)] $H = K_2 + K_{1,3}$;
\item[(v)] $H = K_2 + P_4$;
\item[(vi)] $H = K_2 + K_{3}$;
\item[(vii)] $H = K_{1,4}$;
\item[(viii)] $H = K_{1,3}^+$ -- a $3$-star with one edge subdivided into two;
\item[(ix)] $H = P_{5}$;
\item[(x)] $H = C_4$;
\item[(xi)] $H = K_{3}^+$ -- a triangle with one pendant edge.
\end{itemize}

There are also several `degenerate' cases, in which two of the edges $e,f,g,h$ coincide; we will deal with these cases at the end of the proof. 

We claim that in each of the above listed cases, after some cancellation, the respective contributions can be bounded from above using the terms appearing on the right hand side of~\eqref{eq:Varsq}. As a helpful piece of notation, we denote the number of unlabelled (not necessarily induced) copies of $H$ in $G$ by $N(H)$.

\subsection*{Case (i)} $H = 4 K_2$.

Given a fixed $4$-tuple $(e,f,g,h)$ forming $H$, it follows by~\eqref{eq:Cov} that
\begin{align*}
\Cov(e,f,g,h) &= \E(X_e X_f X_g X_h) - 4 \E(X_e X_f X_g) p + 6 \E(X_e X_f) p^2 - 4 \E(X_e) p^3 + p^4\\
&= \frac{1}{n^8} \left[(k)_8 - 4 (k)_6 (k)_2 + 6 (k)_4 ((k)_2)^2 - 3 ((k)_2)^4 \right] (1 + O_k(1/n))\\
&= \frac{1}{n^8} [k^8 - 28 k^7 - 4 k^8 + 64 k^7 + 6 k^8 - 48 k^7 - 3 k^8 + 12 k^7 + O(k^6)]\\
&= O \left(\frac{k^6}{n^8} \right).
\end{align*}
Taking the sum over all such tuples $(e,f,g,h)$, we obtain an overall contribution of at most
\[O \left(\frac{m^4k^6}{n^8} \right) = O \left(m^2 \cdot \frac{k^4}{n^4} \right) = O \left(Var(X)^2 \right),\] 
where the penultimate equality holds since $m = O(n^2/k)$ follows from Lemma~\ref{lem:maxdeg} and the last equality holds by~\eqref{eq:Varsq}. 

\subsection*{Case (ii)} $H = 2 K_2 + K_{1,2}$.

Let $e$ and $f$ denote the two edges that share a vertex. It follows by~\eqref{eq:Cov} and a straightforward calculation that
\begin{align*}
\Cov(e,f,g,h) &= \left[\E(X_e X_f X_g X_h) - \E(X_e X_f X_g) p - \E(X_e X_f X_h) p + \E(X_e X_f) p^2 \right] (1 + O_k(1/n))\\
&= \frac{1}{n^7} \left[(k)_7 - 2 (k)_5 (k)_2 + (k)_3 ((k)_2)^2 \right] (1 + O_k(1/n))\\
&= \frac{1}{n^7} \left[k^7 - 21 k^6 - 2 k^7 + 22 k^6 + k^7 - 5 k^6 + O(k^5) \right] \\
&= \frac{1}{n^7}(-4 k^6 + O(k^5)) < 0.
\end{align*}
Thus, we can ignore such tuples $(e,f,g,h)$ as their contribution is negative. 

\subsection*{Case (iii)} $H = 2 K_{1,2}$.

\noindent
A straightforward calculation shows that $\Cov(e,f,g,h) = (1+o(1))(\E(X_e X_f X_g X_h)) = O({k^6}/{n^6})$.
Moreover, the total number of such $4$-tuples $(e,f,g,h)$ is at most $4! \cdot S^2$. Hence,
\[\sum_{(e,f,g,h) \cong H} \Cov(e,f,g,h) = O \left(S^2 \cdot \frac{k^6}{n^6} \right) = O \left(Var(X)^2 \right), \]
where the last equality holds by~\eqref{eq:Varsq}.

\subsection*{Case (iv)} $H = K_2 + K_{1,3}$. 

\noindent
It follows from Claim~\ref{obs:cov} that $\Cov(e,f,g,h) = O({k^6}/{n^6})$.
Note that a copy of $K_{1,3}$ in $G$ can be viewed as a copy of $P_3$ with an additional edge attached to its middle vertex. This implies that $N(K_{1,3}) = O(N(P_3) \cdot \Delta(G)) = O(S n/k)$, where the last equality holds by Lemma~\ref{lem:maxdeg}. Therefore, 
\[\sum_{(e,f,g,h) \cong H} \Cov(e,f,g,h) = O \left(m S \cdot k^5/n^5 \right) = O \left(Var(X)^2 \right),\]
where the last equality holds by~\eqref{eq:Varsq}. 

\noindent
\subsection*{Case (v)} $H = K_2 + P_4$.

\noindent
This case is analogous to Case (iv) with $P_4$ in place of $K_{1,3}$. Since a copy of $P_4$ in $G$ can be viewed as a copy of $P_3$ with another vertex connected by an edge to one of its end vertices, we have $N(P_4) = O(N(P_3) \cdot \Delta(G)) = O(S n/k)$, where the last equality holds by Lemma~\ref{lem:maxdeg}. We conclude that
$$
\sum_{(e,f,g,h) \cong H} \Cov(e,f,g,h) =  O\left(m \cdot N(P_4) \cdot k^6/n^6 \right) = O \left(m S \cdot k^5/n^5 \right) = O \left(Var(X)^2 \right),
$$
where the last equality holds by~\eqref{eq:Varsq}.

\subsection*{Case (vi)} $H = K_2 + K_3$.
\noindent
This is similar to cases (iv) and (v). It follows from Claim~\ref{obs:cov} that 
$$
\Cov(e,f,g,h) = O({k^5}/{n^5}).
$$
Hence,
$$
\sum_{(e,f,g,h) \cong H} \Cov(e,f,g,h) = O \left(m \cdot N(K_3) \cdot \frac{k^5}{n^5} \right) = O \left(m S \cdot k^5/n^5 \right) = O \left(Var(X)^2 \right),
$$
where the second equality holds since $N(K_3) = O(S)$ and the last equality holds by~\eqref{eq:Varsq}.

\subsection*{Case (vii)} $H = K_{1,4}$.

\noindent
It follows from Claim~\ref{obs:cov} that
\begin{equation} \label{eq:vii}
\sum_{(e,f,g,h) \cong H} \Cov(e,f,g,h) = O \left( N(K_{1,4}) \cdot \frac{k^5}{n^5} \right).
\end{equation}

Note that 
\[N(K_{1,4}) = O \left(\sum_{v \in V(G)} d(v)^4 \right)=  O \left(\Delta(G)^2 \cdot \sum_{v \in V(G)} d(v)^2 \right) = O \left(\frac{n^2}{k^2} \cdot S \right) = O(mS),\]
where the penultimate equality holds by Lemma~\ref{lem:maxdeg} and the last equality holds by Lemma~\ref{lem:avgdeg}. We conclude that
$$
\sum_{(e,f,g,h) \cong H} \Cov(e,f,g,h) = O \left(m S \cdot \frac{k^5}{n^5} \right) = O \left(Var(X)^2 \right), 
$$
where the last equality holds by~\eqref{eq:Varsq}. 

\subsection*{Case (viii)} $H = K_{1,3}^+$.

\noindent
This is very similar to Case (vii) with $K_{1,3}^+$ in place of $K_{1,4}$. Since a copy of $K_{1,3}^+$ in $G$ can be viewed as a copy of $P_3$ with two leaves attached to one of its end vertices, it follows that 
$$
N \left(K_{1,3}^+ \right) = O \left(S \cdot \Delta(G)^2 \right) = O(m S).
$$ 
Therefore
$$
\sum_{(e,f,g,h) \cong H} \Cov(e,f,g,h) = O \left(N \left(K_{1,3}^+ \right) \cdot \frac{k^5}{n^5} \right) = O \left(m S \cdot \frac{k^5}{n^5} \right) = O \left(Var(X)^2 \right), 
$$
where the last equality holds by~\eqref{eq:Varsq}.

\subsection*{Case (ix)} $H = P_5$.

\noindent
This is again very similar to Case (vii) with $P_5$ in place of $K_{1,4}$. Since a copy of $P_5$ in $G$ can be viewed as a copy of $P_3$ with a leaf attached to each of its end vertices, it follows that 
$$
N(P_5) = O \left(S \cdot \Delta(G)^2 \right) = O(m S).
$$ 
Therefore
$$
\sum_{(e,f,g,h) \cong H} \Cov(e,f,g,h) = O \left(N(P_5) \cdot \frac{k^5}{n^5} \right) = O \left(m S \cdot \frac{k^5}{n^5} \right) = O \left(Var(X)^2 \right), 
$$
where the last equality holds by~\eqref{eq:Varsq}.

\subsection*{Case (x)} $H = C_4$.
It follows from Claim~\ref{obs:cov} that
\[
\sum_{(e,f,g,h) \cong H} \Cov(e,f,g,h) = O \left(N(C_4) \cdot \frac{k^4}{n^4} \right).
\]
It is evident that $N(C_4) = O(N(P_4))$. Since, moreover, we can view $P_4$ as $P_3$ with an additional edge appended to one of its leaves, it follows that 
$$
N(P_4) = O \left(N(P_3) \cdot \Delta(G) \right) = O \left(S \cdot \Delta(G) \right) = O \left(S \cdot \frac{n}{k} \right),
$$  
where the last equality holds by Lemma~\ref{lem:maxdeg}. Therefore  
\begin{align*}
\sum_{(e,f,g,h) \cong H} \Cov(e,f,g,h) &= O \left(N(C_4) \cdot \frac{k^4}{n^4} \right) = O \left(S \cdot \frac{k^3}{n^3} \right) \\
&= O \left(m S \cdot \frac{k^5}{n^5} \right) = O \left(Var(X)^2 \right),
\end{align*}
where the penultimate equality holds by Lemma~\ref{lem:avgdeg} and the last equality holds by~\eqref{eq:Varsq}.

\subsection*{Case (xi)} $H = K_{3}^+$.
\noindent
Using the calculations made in Case (x) we obtain
$$
N \left(K_3^+ \right) = O(N(P_4)) = O \left(S \cdot \frac{n}{k} \right).
$$
It thus follows by Claim~\ref{obs:cov} that
$$
\sum_{(e,f,g,h) \cong H} \Cov(e,f,g,h) = O \left(N \left(K_{3}^+ \right) \cdot \frac{k^4}{n^4} \right) = O \left(S \cdot \frac{k^3}{n^3} \right) = O \left(m S \cdot \frac{k^5}{n^5} \right) = O \left(Var(X)^2 \right),
$$
where the last equality holds by~\eqref{eq:Varsq}.

\subsection*{The degenerate cases} We now consider the cases where some edge of $\{e,f,g,h\}$ appears more than once. In those cases, the corresponding graph $H$ has at most three edges.
\textbf{\begin{itemize}
\item[(xii)] $H = 3 K_2$
\item[(xiii)] $H = K_2 + K_{1,2}$ 
\item[(xiv)] $H = K_{1,3}$
\item[(xv)] $H = P_4$
\item[(xvi)] $H = K_3$
\item[(xvii)] $H=K_{1,2}$
\item[(xviii)] $H=2K_2$
\item[(xix)]  $H=K_2$
\end{itemize}}

\subsection*{Case (xii)} $H = 3K_2$.
\noindent
Without loss of generality, suppose that $e,f,g$ are independent and that $h=e$. Since, clearly $X_e^2 = X_e$, we obtain

\begin{align*}
& \Cov(e,f,g,h)=\E((X_e-p)^2(X_f-p)(X_g-p)) \\
&= \left(\E(X_e X_f X_g) - p \E(X_e X_f) - p \E(X_e X_g) + p^2 \E(X_e) \right) \left(1 + O_k(1/n) \right) \\
&= \frac{1}{n^6} \left((k)_6 - 2 (k)_2 (k)_4 + ((k)_2)^3 \right) \left(1 + O_k(1/n) \right) \\
&= \frac{-4k^5 + O(k^4)}{n^6} < 0.
\end{align*}
Thus, the total contribution of all such tuples is negative.

\subsection*{Case (xiii)} $H = K_2 + K_{1,2}$.
\noindent
It follows from Claim~\ref{obs:cov} that
$$
\sum_{(e,f,g,h) \cong H} \Cov(e,f,g,h) = O \left(m N(K_{1,2}) \cdot \frac{k^5}{n^5} \right) = O \left(m S \cdot \frac{k^5}{n^5} \right) = O \left(Var(X)^2 \right),
$$
where the last equality holds by~\eqref{eq:Varsq}.

\subsection*{Case (xiv)} $H = K_{1,3}$.

\noindent It follows from Claim~\ref{obs:cov} that
\begin{align*}
\sum_{(e,f,g,h) \cong H} \Cov(e,f,g,h) &= O \left(N(K_{1,3}) \cdot \frac{k^4}{n^4} \right) = O \left(\Delta(G) \cdot N(K_{1,2}) \cdot \frac{k^4}{n^4} \right) \\ 
&= O \left(S \cdot \frac{k^3}{n^3} \right) = O \left(m S \cdot \frac{k^5}{n^5} \right) = O \left(Var(X)^2 \right),
\end{align*}
where the third equality holds by Lemma~\ref{lem:maxdeg}, the fourth equality holds by Lemma~\ref{lem:avgdeg}, and the last equality holds by~\eqref{eq:Varsq}.

\subsection*{Case (xv)} $H = P_4$.

\noindent It follows from Claim~\ref{obs:cov} that
\begin{align*}
\sum_{(e,f,g,h) \cong H} \Cov(e,f,g,h) = O \left(N(P_4) \cdot \frac{k^4}{n^4} \right)
= O \left(S \cdot \frac{k^3}{n^3} \right) = O \left(m S \cdot \frac{k^5}{n^5} \right) = O \left(Var(X)^2 \right),
\end{align*}
where the second equality was proved in Case (x) and the last equality holds by~\eqref{eq:Varsq}.

\subsection*{Case (xvi)} $H = K_3$.
\noindent It follows from Claim~\ref{obs:cov} that
\begin{align*}
\sum_{(e,f,g,h) \cong H} \Cov(e,f,g,h) = O \left(N(K_3) \cdot \frac{k^3}{n^3} \right) = O \left(S \cdot \frac{k^3}{n^3} \right) = O \left(m S \cdot \frac{k^5}{n^5} \right) = O \left(Var(X)^2 \right),
\end{align*}
where the last equality holds by~\eqref{eq:Varsq}.

\subsection*{Case (xvii)} $H=K_{1,2}$.
\noindent It follows from Claim~\ref{obs:cov} that
\begin{align*}
\sum_{(e,f,g,h) \cong H} \Cov(e,f,g,h) = O \left(N(K_{1,2}) \cdot \frac{k^3}{n^3} \right) = O \left(S \cdot \frac{k^3}{n^3} \right) = O \left(m S \cdot \frac{k^5}{n^5} \right) = O \left(Var(X)^2 \right),
\end{align*}
where the last equality holds by~\eqref{eq:Varsq}.

\subsection*{Case (xviii)} $H=2K_2$.
\noindent It follows from Claim~\ref{obs:cov} that
\begin{align*}
\sum_{(e,f,g,h) \cong H} \Cov(e,f,g,h) = O \left(N(2K_2) \cdot \frac{k^4}{n^4} \right) = O \left(m^2 \cdot \frac{k^4}{n^4} \right) = O \left(Var(X)^2 \right),
\end{align*}
where the last equality holds by~\eqref{eq:Varsq}.

\subsection*{Case (xix)} $H=K_2$.
\noindent It follows from Claim~\ref{obs:cov} that
\begin{align*}
\sum_{(e,f,g,h) \cong H} \Cov(e,f,g,h) = O \left(m \cdot \frac{k^2}{n^2} \right) = O \left(m^2 \cdot \frac{k^4}{n^4} \right) = O\left(Var(X)^2 \right),
\end{align*}
where the last equality holds by~\eqref{eq:Varsq}.

To conclude, we have considered every possible case and in each one we have shown that its respective contribution to $\E \left((X-\mu)^4 \right)$ is of
order of magnitude $O \left(Var(X)^2 \right)$. Since the number of cases is constant, this completes the proof of Lemma~\ref{lem:truth} and thus also of Theorem~\ref{thm:main}.
\end{proof}

\section{Very small inducibility for almost all $\ell$} \label{sec:poly} 

Our aim in this section is to prove Theorem~\ref{thm:poly}. Let $k$ and $\ell$ be positive integers for which $\min \left\{k, \binom{k}{2} - \ell \right\} = \Omega \left(k^2 \right)$. Since, $\ind(k,\ell)$ is defined as $\lim_{n \to \infty} I(n, k, \ell)$ and the latter sequence is monotone decreasing, it suffices to show that $I(2k, k, \ell) = O \left(k^{-0.1} \right)$. 
Suppose then that we have a $2k$-vertex graph $G = (V,E)$, in which we are selecting a $k$-vertex set $A \subseteq V$ uniformly at random. As in the proofs in previous sections, we may assume that $G$ maximizes $\Pb(X_{G,k} = \ell)$ amongst all $2k$-vertex graphs.
\begin{claim}\label{cl:eGquadratic}
$\min \{e(G), e(\overline{G})\} = \Omega(k^2)$.
\end{claim}

\begin{proof}
%
Since $|G| = 2k$ and $G$ contains an induced subgraph of order $k$ and size $\ell = \Omega \left(k^2 \right)$ (as, clearly, $I(2k, k, \ell) > 0$), we have that $e(G) = \Omega \left(k^2 \right)$.
The statement for $\overline{G}$ follows analogously.
\end{proof}

A pair of distinct vertices $\{u,v\} \subseteq V$ will be called \emph{distinguished} if $|N_G(u) \triangle N_G(v)| = \Theta(k)$. Let $D = D(G)$ be the set of all distinguished pairs in $G$.
\begin{claim} \label{cl:distinguished}
$|D| = \Theta(k^2)$.
\end{claim}

\begin{proof}
It is evident that $|D| = O(k^2)$. To see that $|D| = \Omega(k^2)$, note that 
\begin{align*}
\sum_{\{u,v\} \in \binom{V}{2}} |N(u) \triangle N(v)| &= \sum_{\{u,v\} \in \binom{V}{2}} (d(u) + d(v) - 2|N(u)\cap N(v)|)\\
&= (|V| - 1) \sum_{v \in V} d(v) - 2 \sum_{v \in V} \binom{d(v)}{2}\\
&= 2 \left[(2k-1) e(G) - \sum_{v \in V} \binom{d(v)}{2} \right].
\end{align*}
By a theorem of Ahlswede and Katona~\cite{AK}, for a graph $H$ of given order and size the quantity $\sum_v \binom{d(v)}{2}$, which corresponds to the number of cherries $K_{1,2}$, is maximized when either $H$ or its complement is a union of a clique and an independent set. In either case it follows by Claim~\ref{cl:eGquadratic} and a straightforward calculation that 
$$
\sum_{\{u,v\} \in \binom{V}{2}} |N(u) \triangle N(v)| \geq \sum_{\{u,v\} \in \binom{V(H)}{2}} |N_H(u) \triangle N_H(v)| = \Omega(k^3).
$$ 
Since the maximum degree of $G$ is $O(k)$, we conclude that $D(G) = \Omega(k^2)$, as claimed.
\end{proof}

\begin{claim}\label{cl:concentr}
For any graph $F$ with $2k$ vertices and $\Omega \left(k^2 \right)$ edges, and any integer $t$ satisfying $k - \sqrt{k} \leq t \leq k + \sqrt{k}$, we have 
$$
\Pb\left(X_{F,t} \leq \frac{e(F)}{5}\right) = o \left(k^{-1} \right).
$$
\end{claim}

\begin{proof}
Let $B \subseteq V(F)$ be a random set of size $t$ and let $\overline{B} = V(F) \setminus B$ be its complement. Fix an arbitrary vertex $v \in V$. It follows by standard tail estimates for the hypergeometric distribution that  
$$
\Pb\left(|d(v, B) - d(v, {\overline{B}})| > k^{0.9}\right) = o \left(k^{-2} \right)
$$
(in fact, the right hand side is exponentially small in a positive constant power of $k$, see for example~\cite[Theorem 2.10]{JLR}). Hence, by the union bound, with probability $1 - o \left(k^{-1} \right)$, it holds that $|d(v, B) - d(v, {\overline{B}})| \leq k^{0.9}$ for every $v \in V(F)$. Summing over all vertices in $B$, we obtain 
$$
|2 e(F[B]) - e(F[B, \overline{B}])| \leq \sum_{v \in B} |d(v, B) - d(v, {\overline{B}})| \leq 2 k^{1.9} = o(e(F)).
$$
Similarly, summing over all vertices in $\overline{B}$, we obtain 
$$
|2 e(F[\overline{B}]) - e(F[B, \overline{B}])| = o(e(F)).
$$
Hence, with probability $1 - o \left(k^{-1} \right)$ we have
\begin{align*}
e(F[B]) &= e(F[\overline{B}]) + o(e(F)) = \frac{1}{2} \cdot e(F[B, \overline{B}]) + o(e(F))\\
&= \left(1/4 + o(1) \right) e(F).
\end{align*}
In particular, $X_{F,t} = e(F[B]) \geq {e(F)}/{5}$ holds with probability $1 - o \left(k^{-1} \right)$.
\end{proof}

Now, let us sample the set $A$ in two steps as follows. We first sample a set $A_0 \subseteq V$ of size $k-k^{0.2}$ uniformly at random, then we sample a set $A_1 \subseteq V \setminus A_0$ of size $k^{0.2}$ in a manner which will be specified later, and finally we set $A = A_0 \cup A_1$. Note that $A_1$ will be sampled in a way which will ensure that the resulting set $A$ will indeed be chosen uniformly amongst all subsets of $V(G)$ of size $k$. 

Given $A_0$, a distinguished pair $\{u,v\}$ is said to be \emph{bad} if $|d(u, A_0) - d(v, A_0)| \leq k^{0.4}$, and \emph{good} otherwise. Our next claim shows that, with sufficiently high probability, most distinguished pairs are good. 
\begin{claim} \label{cl:FewBadPairs}
With probability $1 - O \left(k^{-0.1} \right)$ there are at most $|D|/6$ bad pairs.
\end{claim}

\begin{proof}
Fix some distinguished pair $\{u,v\}$ and let $X_{uv} = 1$ if $\{u,v\}$ is bad and $X_{uv} = 0$ otherwise. Let $Y_{uv} = |A_0 \cap (N(u) \triangle N(v))|$ and let $s = |N(u) \triangle N(v)|$; recall that $\{u,v\}$ is a distinguished pair and thus $s = \Theta(k)$. Then
\begin{align} \label{eq::LTP}
\Pb(X_{uv} = 1) &= \sum_{t = 0}^s \Pb(X_{uv} = 1 | Y_{uv} = t) \cdot \Pb(Y_{uv} = t) \nonumber \\ 
&= \sum_{t = \varepsilon k}^{0.9 s} \Pb(X_{uv} = 1 | Y_{uv} = t) \cdot \Pb(Y_{uv} = t) + O \left(k^{-1} \right),
\end{align}
where the first equality holds by the law of total probability and the second equality holds for a sufficiently small constant $\varepsilon > 0$ since, by standard tail estimates for the hypergeometric distribution, we have that $\Pb(Y_{uv} \leq \varepsilon k) = O \left(k^{-1} \right)$ and $\Pb(Y_{uv} \geq 0.9 s) = O \left(k^{-1} \right)$.  

Assume without loss of generality that $|N(u) \setminus N(v)| \geq |N(v) \setminus N(u)|$; observe that by the definition of a distinguished pair it thus follows that $r := |N(u) \setminus N(v)| = \Theta(k)$. Let $Z_{uv} = |A_0 \cap (N(u) \setminus N(v))|$. Observe that for any $\varepsilon k \leq t \leq s$ we have 
\begin{equation} \label{eq::Zuv}
\Pb(X_{uv} = 1 | Y_{uv} = t) \leq \sum_{i = (t - k^{0.4})/2}^{(t + k^{0.4})/2} \Pb(Z_{uv} = i | Y_{uv} = t). 
\end{equation}
Note that $Z_{uv} | Y_{uv} = t$ is a hypergeometric random variable with parameters $s$, $t$ and $r$. Since, moreover, $\mathbb{E}(Z_{uv}) = r t/s = \Theta(k)$, a straightforward calculation shows that for every $(t - k^{0.4})/2 \leq i \leq (t + k^{0.4})/2$ we have 
\begin{equation} \label{eq::HGanticon}
\Pb(Z_{uv} = i | Y_{uv} = t) = O \left(k^{-1/2} \right).
\end{equation}
Combining~\eqref{eq::LTP}, \eqref{eq::Zuv}, and~\eqref{eq::HGanticon} we obtain  
\begin{equation} \label{eq::LowProbForBadPair}
\Pb(X_{uv} = 1) = O \left(k^{- 0.1} \right).
\end{equation}

Now, let $X = \sum_{\{u,v\} \in D} X_{uv}$ be the random variable which counts the number of bad pairs in $D$. It follows by~\eqref{eq::LowProbForBadPair} that 
$$
\mathbb{E}(X) = \sum_{\{u,v\} \in D} \mathbb{E}(X_{uv}) = O \left(k^{1.9} \right).
$$
Applying Markov's inequality to $X$ we conclude that $\Pb(X \geq |D|/6) = O \left(k^{-0.1} \right)$.
\end{proof}

\begin{claim}\label{cl:XandY}
With probability $1 - O \left(k^{-0.1} \right)$ the set $A_0$ will have the following property: there exist disjoint sets $X, Y \subseteq V \setminus A_0$ such that $|X| = |Y| = \Theta(k)$ and $d(x,A_0) - d(y,A_0) \geq k^{0.4}$ holds for every $x \in X$ and $y \in Y$.
\end{claim}

\begin{proof}
Recall that $|D| = \Theta \left(k^2 \right)$ holds by Claim~\ref{cl:distinguished}. By considering an auxiliary graph $F$ with $V(F) = V$ and $E(F) = \{uv : \{u,v\} \in D\}$ to which we apply Claim~\ref{cl:concentr}, we infer that, with probability $1 - o(k^{-1})$, at least $|D|/5$ of the distinguished pairs are disjoint from $A_0$. On the other hand, it follows by Claim~\ref{cl:FewBadPairs} that, with probability $1 - O \left(k^{-0.1} \right)$, at most $|D|/6$ of the distinguished pairs are bad. We conclude that, with probability $1 - O \left(k^{-0.1} \right)$, there is a set $D' \subseteq D\cap \binom{V \setminus A_0}{2}$ of size $c k^2$, for a constant $c>0$, such that $|d(u, A_0) - d(v, A_0)| \geq k^{0.4}$ for every $\{u,v\} \in D'$.

Let $(u_1, \ldots u_{k+k^{0.2}})$ be an ordering of the vertices of $V\setminus A_0$ by non-increasing order of degrees into $A_0$, that is, $d(u_i, A_0) \geq d(u_j, A_0)$ for every $1 \leq i < j \leq k+k^{0.2}$. Let $q = c k/3$, let $X = \{u_1, \ldots, u_q\}$, and let $Y = \{u_{k+k^{0.2}-q+1}, \ldots, u_{k+k^{0.2}}\}$. Observe that $X \cap Y = \emptyset$ and $|X| = |Y| = q = \Theta(k)$. It thus remains to prove that $d(x, A_0) - d(y, A_0) \geq k^{0.4}$ holds for every $x \in X$ and $y \in Y$. Suppose for a contradiction that there exist $x \in X$ and $y \in Y$ such that $d(x, A_0) - d(y, A_0) < k^{0.4}$. Since, by definition, $d(x, A_0) \geq d(u, A_0) \geq d(y, A_0)$ for every $u \in V\setminus (A_0\cup X\cup Y)$, it follows that $\{u,v\} \cap (X \cup Y) \neq \emptyset$ for every $\{u,v\} \in D'$. Therefore
\begin{align*}
c k^2 = |D'| &\leq \binom{|X\cup Y|}{2} + |X \cup Y| \cdot |V \setminus (A_0\cup X\cup Y)| \\
&\leq 2q(q+k+k^{0.2}-2q)\leq 2 c k/3 \cdot k < c k^2,
\end{align*}
which is clearly a contradiction.   
\end{proof}

Let $A_0$ be a set chosen randomly as described above. Let $X$ and $Y$ be disjoint subsets of $V \setminus A_0$ such that $|X| = |Y| = c k$ for some $c > 0$, and $d(x,A_0) - d(y,A_0) \geq k^{0.4}$ holds for every $x \in X$ and $y \in Y$; such sets exist with probability $1 - O \left(k^{- 0.1} \right)$ by Claim~\ref{cl:XandY}. Now, we choose a random set $A_1 \subseteq V \setminus A_0$ of size $k^{0.2}$ as follows. First, we choose $k^{0.2}$ pairwise-disjoint vertex pairs from $V\setminus A_0$ uniformly at random. Then, from each such pair, we choose uniformly at random exactly one element to be in $A_1$; all choices being mutually independent.    

\begin{claim}\label{cl:MK5}
Let $A_1$ be chosen randomly as described above and let $A = A_0 \cup A_1$. Then, for any integer $\ell$, we have $\Pb(e_G(A) = \ell) = O \left(k^{- 0.1} \right)$. 
\end{claim}

\begin{proof}
Let $M$ denote the set of all $k^{0.2}$ randomly chosen pairs. Let $M_{XY}$ denote the set of chosen pairs which have one element in $X$ and one in $Y$, and let $m = |M_{XY}|$. We claim that $m = \Omega \left(k^{0.2} \right)$ with probability $1 - e^{- \Omega(k^{0.2})}$. Indeed, consider choosing the pairs which constitute $M$ one by one. Suppose that we have already chosen $j$ pairs for some $0 \leq j < k^{0.2}$, and now choose the $(j+1)$st pair. Let $M_j$ denote the set of all vertices in the union of these $j$ pairs; clearly $|M_j| = O \left(k^{0.2} \right)$. Let $T_{j+1} = 1$ if the $(j+1)$st pair has one element in $X$ and one in $Y$, and $T_{j+1} = 0$ otherwise. Then
$$
\Pb(T_{j+1} = 1) = \frac{|X \setminus M_j| \cdot |Y \setminus M_j|}{\binom{2k}{2}} \geq \frac{c^2}{3}.
$$ 
Note that $m = \sum_{j=1}^{k^{0.2}} T_j$. Let $Z \sim Bin(k^{0.2}, c^2/3)$ and observe that, by the above calculation, $m$ stochastically dominates $Z$. Hence, using standard bounds on the tail of the binomial distribution, we conclude that
$$
\Pb(m < c^2 k^{0.2}/6) \leq \Pb(Z < c^2 k^{0.2}/6) \leq \Pb(Z < \mathbb{E}(Z)/2) < e^{- c^2 k^{0.2}/24}.
$$   

Let $\{x_1, y_1\}, \{x_2, y_2\}, \ldots, \{x_m, y_m\}$ be the elements of $M_{XY}$, where $\{x_1, \ldots, x_m\} \subseteq X$ and $\{y_1, \ldots, y_m\} \subseteq Y$. Fix any choice of one element from every pair in $M \setminus M_{XY}$. With any choice of one element from every pair in $M_{XY}$, we associate a binary vector $\bar{z} = (z_1, \ldots, z_m)$ in a natural way, namely, for every $1 \leq i \leq m$, $z_i = 1$ if we chose $x_i$ to be in $A_1$ and $z_i = 0$ if we chose $y_i$. For each such vector $\bar{z}$, we denote the resulting random $k$-subset of $V(G)$ by $A_{\bar{z}}$. For two such vectors $\bar{z}$ and $\bar{w}$, we say that $\bar{z} > \bar{w}$ if $z_i \geq w_i$ for every $1 \leq i \leq m$ and $\bar{z} \neq \bar{w}$. We claim that $e_G(A_{\bar{z}}) > e_G(A_{\bar{w}})$ whenever $\bar{z} > \bar{w}$. Indeed, let $1 \leq i \leq m$ be an index for which $z_i = 1$ and $w_i = 0$. Then
$$
e_G(A_{\bar{z}}) - e_G(A_{\bar{w}}) \geq d(x_i, A_0) - d(y_i, A_0) - e_G(A_{\bar{w}} \setminus A_0) \geq k^{0.4} - \binom{k^{0.2}}{2} > 0.  
$$ 
It follows that, for any integer $\ell$, the elements of $\{\bar{z} \in \{0,1\}^m : e_G(A_{\bar{z}}) = \ell\}$ form an anti-chain. Hence, by Sperner's Theorem we conclude that
$$
\Pb(e_G(A) = \ell) \leq \frac{\binom{m}{\lfloor m/2 \rfloor}}{2^m} = O \left(1/\sqrt{m} \right) = O \left(k^{- 0.1} \right).
$$         
\end{proof}  
Claim~\ref{cl:MK5} implies that $I(2k, k, \ell) = O(k^{-0.1})$. Thus indeed $\ind(k, \ell) = O(k^{-0.1})$ and the proof of Theorem~\ref{thm:poly} is complete. Note that the proof of Claim~\ref{cl:MK5} resembles (and was inspired by) the Littlewood-Offord problem and its solution by Erd\H{o}s~\cite{Erd}.
 
\section{When $\ell$ is fixed: Proof of Theorem~\ref{thm:fixed}} \label{sec:fixed}

Our aim in this section is to show that, under certain natural conditions, $X_{G,k}$ exhibits a Poisson-like behaviour. The following proposition makes this precise.


\begin{prop} \label{prop:Brun}
Suppose that $1\leq \ell \ll k \ll n$ are integers and that $G$ is a graph with $n$ vertices and $m$ edges. Suppose that $\lim_{k \to \infty} m \cdot \frac{k^2}{n^2} = \mu$ for some constant $\mu = \mu(\ell) > 0$ and that $\Delta := \Delta(G) = o_k(n/k)$. Then 
$$
\Pb(X_{G,k} = \ell) = (1 + o_k(1)) e^{-\mu} \cdot \frac{\mu^\ell}{\ell!}.
$$
\end{prop}

\begin{proof}
Throughout this proof the $o(\cdot)$-notation will solely refer to $o_k(\cdot)$; hence we omit the subscript $k$ for readability. 
By Brun's Sieve (see, e.g., Theorem 8.3.1 in~\cite{AS}), in order to prove that $\Pb(X_{G,k} = \ell) = (1 + o(1)) e^{-\mu} \cdot \frac{\mu^\ell}{\ell!}$, it suffices to show that $\mathbb{E} \left[\binom{X}{r} \right] = (1 + o(1)) \cdot \frac{\mu^r}{r!}$ holds for every fixed integer $1 \leq r \leq r(\ell)$.     

Let $e_1, \ldots, e_m$ be an arbitrary ordering of the edges of $G$. For every $1 \leq j \leq m$, let $X_j = 1$ if both endpoints of $e_j$ are in the random $k$-set $A$, and $X_j = 0$ otherwise. Observe that $X_{G,k} = \sum_{j=1}^m X_j$. In particular, 
$$
\mathbb{E}(X_{G,k}) = \sum_{j=1}^m \mathbb{E}(X_j) = m \cdot \frac{\binom{n-2}{k-2}}{\binom{n}{k}} = m \cdot \frac{k (k-1)}{n (n-1)}.
$$
It thus follows by the definition of $\mu$ that $\mathbb{E} \left[\binom{X}{1} \right] = \mathbb{E}(X) = (1+o(1)) \mu$.

Now, fix some positive integer $r$. Let $1 \leq j_1 < \ldots < j_r \leq m$ be arbitrary indices. Assume first that $e_{j_1}, \ldots, e_{j_r}$ form a matching in $G$. Then
$$
\Pb(X_{j_1} = 1 \wedge \ldots \wedge X_{j_r} = 1) = \frac{\binom{n - 2r}{k - 2r}}{\binom{n}{k}} = (1 + o(1)) \left(\frac{k}{n} \right)^{2r}.
$$
Moreover, the number of ways to choose indices $1 \leq j_1 < \ldots < j_r \leq m$ for which $e_{j_1}, \ldots, e_{j_r}$ form a matching in $G$ is $(1 + o(1)) \binom{m}{r}$ (trivially, it is at most $\binom{m}{r}$). Indeed, by our assumption on $\Delta$, this number is at least 
$$
\frac{1}{r!} \prod_{t=0}^{r-1} (m - 2 t \Delta) \geq \frac{1}{r!} \prod_{t=0}^{r-1} (m - 2 t m k/n) = \frac{m^r}{r!} \cdot (1 + o(1)) e^{- 2 k \binom{r}{2}/n} = (1 + o(1)) \binom{m}{r},
$$   
where the last equality holds since, by assumption, $n$ is sufficiently large with respect to $k$ and $r$.

Next, assume that $e_{j_1}, \ldots, e_{j_r}$ do not form a matching in $G$. Let $H$ be the graph whose edges are $e_{j_1}, \ldots, e_{j_r}$ and whose vertices are the endpoints of $e_{j_1}, \ldots, e_{j_r}$. Let $C_1, \ldots, C_t$ denote the connected components of $H$, and let $c_i = |C_i|$ for every $1 \leq i \leq t$. Then
$$
\Pb(X_{j_1} = 1 \wedge \ldots \wedge X_{j_r} = 1) = \frac{\binom{n - \sum_{i=1}^t c_i}{k - \sum_{i=1}^t c_i}}{\binom{n}{k}} = (1 + o(1)) \left(\frac{k}{n} \right)^{\sum_{i=1}^t c_i}.
$$  

Assume without loss of generality that $c_1 \geq \ldots \geq c_t$. Since $H$ is not a matching, we must have $c_1 \geq 3$. For all integers $a \geq 2$ and $b \geq a-1$, the number of ways to choose $b$ edges of $G$ that form a connected component on $a$ vertices, is at most   
$$
m (a-1)! \Delta^{a-2} \binom{a}{2}^{b-a+1} = 
\begin{cases}
O(m) & \textrm{if } a = 2 \\
m \cdot o\left[(n/k)^{a-2}\right] & \textrm{if } a > 2
\end{cases}
$$
Indeed, we begin by choosing an arbitrary edge of $G$ (in $m$ ways), then, one by one, we choose $a-2$ additional edges to form a tree on $a$ vertices (in at most $2 \Delta \cdot 3 \Delta \cdot \ldots \cdot (a-1) \Delta = (a-1)! \Delta^{a-2}$ ways), and then we choose the remaining $b-a+1$ edges such that their endpoints are among the $a$ vertices of the tree. 

Hence, the total number of ways to choose $r$ edges of $G$ that form a graph $H$ consisting of connected components of orders $c_1 \geq \ldots \geq c_t \geq 2$, where $c_1 \geq 3$, is at most
$$
\prod_{i=1}^t m \cdot o \left[\left(\frac{n}{k} \right)^{c_i-2} \right] = m^t \cdot o\left[\left(\frac{n}{k} \right)^{\sum_{i=1}^t c_i - 2t}\right].
$$

For every $1 \leq t \leq r$, let $L_t$ denote the set of all integer vectors $(c_1, \ldots, c_t)$ such that $c_1 \geq \ldots \geq c_t \geq 2$, $c_1 \geq 3$, and $\sum_{i=1}^t c_i < 2r$. Combining everything together we conclude that 
\begin{eqnarray*}
\mathbb{E} \left[\binom{X}{r} \right] &=& (1 + o(1)) \binom{m}{r} \cdot (1 + o(1)) \left(\frac{k}{n} \right)^{2r} \\ 
&+& \sum_{t=1}^r \sum_{(c_1, \ldots, c_t) \in L_t} m^t \cdot o\left[\left(\frac{n}{k} \right)^{\sum_{i=1}^t c_i - 2t}\right] \cdot (1 + o(1)) \left(\frac{k}{n} \right)^{\sum_{i=1}^t c_i} \\
&=& (1 + o(1)) \left(\frac{1}{r!} \left(\frac{m k^2}{n^2} \right)^r + \sum_{t=1}^r \sum_{(c_1, \ldots, c_t) \in L_t} o \left[\left(\frac{m k^2}{n^2} \right)^t \right] \right) \\
&=& (1 + o(1)) \left[\frac{\mu^r}{r!} + o(\mu^r) \right] = (1 + o(1)) \cdot \frac{\mu^r}{r!}, 
\end{eqnarray*}
where the penultimate equality holds since $r$, $t$ and $\mu > 0$ are fixed. 
\end{proof}

The first consequence of Proposition~\ref{prop:Brun} is that it provides a whole plethora of constructions demonstrating that $\ind(k,1) \geq 1/e + o_k(1)$. Indeed, let $n$ be a sufficiently large integer and let $G$ be any graph with $n$ vertices, $(1 + o_k(1)) \left(n^2/k^2 \right)$ edges, and maximum degree $o_k(n/k)$. Then, $G$ satisfies  the assumptions of Proposition~\ref{prop:Brun} with $\mu = 1$, implying that $\ind(k,1) \geq 1/e + o_k(1)$. This rich family of constructions includes, in particular, the random graphs $G \left(n, \binom{k}{2}^{-1} \right)$ (with high probability) which were mentioned in the introduction, and the pairwise disjoint union of $\binom{k}{2}$ cliques, on $n \binom{k}{2}^{-1}$ vertices each.

Moreover, by combining Proposition~\ref{prop:Brun} with some of our previous arguments as well as some new ones, we can prove Theorem~\ref{thm:fixed}.

\subsection{The $1/2$-bound for $\ell=1$}

Let $k$ be a sufficiently large integer, let $a = \ind(k,1)$, and let $G$ be a graph on $n$ vertices, attaining the maximum density of induced one-edged graphs, that is, $\Pb(X_{G,k} = 1) = a + o_n(1)$.

Our aim is to show that either $a \leq 1/2 + o_k(1)$, or that the conditions of Proposition~\ref{prop:Brun} are satisfied. In the latter case we will be done, as applying Proposition~\ref{prop:Brun} would imply $a \leq 1/e + o_k(1)$.

\begin{claim} \label{cl:DeltaOnk}
$\Delta(G) \leq 10n/k$.
\end{claim}

\begin{proof}
Let $v$ be a vertex of maximum degree in $G$ and suppose for a contradiction that $d(v) >10n/k$. It follows by Lemma~\ref{lem:zykov} and the discussion succeeding the proof of Proposition~\ref{prop:Brun} that 
$$
\Pb(X_{G,k} = 1 \mid v \in A) = a + o_n(1) \geq 1/e + o_k(1).
$$ 
However, a straightforward calculation which uses the Poisson-approximation shows that 
\begin{align*}
\Pb(X_{G,k} > 1 \mid v \in A) &\geq \Pb(|A \cap N_G(v)| > 1 \mid v \in A) > 1 - e^{-10} - 10 e^{-10} + o_k(1)\\
& > 1 -  \Pb(X_{G,k} = 1 \mid v \in A),
\end{align*}
which is an obvious contradiction.
\end{proof}
\begin{claim}\label{cl:DeltaOmegank}
If $\Delta(G) = o_k(n/k)$, then $a \leq 1/e + o_k(1)$. 
\end{claim}

\begin{proof}
Recall that $e(G) = \Omega \left(n^2/k^2 \right)$ holds by Lemma~\ref{lem:avgdeg}. Assume first that $e(G) \leq 20 n^2/k^2$. By compactness we may assume that $m = (1 + o_k(1)) \mu \frac{n^2}{k^2}$ for some constant $0 < \mu \leq 20$, whereby the conditions of Proposition~\ref{prop:Brun} are satisfied. Applying it yields 
$$
a + o_n(1) = \Pb(X_{G,k} = 1) = (1 + o_k(1)) \mu e^{-\mu} \leq 1/e + o_k(1).
$$   

Assume then that $e(G) > 20 n^2/k^2$. Let $G'$ be an arbitrary subgraph of $G$ with $20 n^2/k^2$ edges. Since $\Delta'(G) \leq \Delta(G) = o_k(n/k)$, the graph $G'$ satisfies the conditions of Proposition~\ref{prop:Brun} with $\mu = 20$. Therefore
$$
a + o_n(1) = \Pb(X_{G,k} = 1) \leq \Pb(X_{G',k} \leq 1) = 21 e^{-20} + o_k(1) \leq 1/e + o_k(1).
$$   
\end{proof}

By Claims~\ref{cl:DeltaOnk} and~\ref{cl:DeltaOmegank} and by compactness, we may assume that $\Delta(G) = c n/k$ for some constant $c > 0$.

Let $v$ be a vertex of maximum degree in $G$, let $Q = N_G(v)$, and let $R = V(G) \setminus (Q \cup \{v\})$. By the Poisson-approximation of the hypergeometric distribution, with probabilities of approximately $e^{-c}$ and $c e^{-c}$, respectively, $A$ will contain exactly $0$ or $1$ vertices from $Q$. 
Therefore, invoking Lemma~\ref{lem:zykov}, we obtain
\begin{align} \label{eq:main2}
a + o_n(1) &= \Pb(X_{G,k} = 1 | v \in A) = \Pb(X_{G,k} = 1 | v \in A \textrm{ and } |A \cap Q| = 0) \cdot \Pb(|A \cap Q| = 0 | v \in A)\nonumber \\
&+ \Pb(X_{G,k} = 1 | v \in A \textrm{ and } |A \cap Q| = 1) \cdot \Pb(|A \cap Q| = 1 | v \in A)\nonumber \\
&\leq (1+o_k(1)) \cdot (\Pb(X_{G[R],k-1} = 1) \cdot e^{-c} + \Pb(X_{G[R],k-2} = 0) \cdot c e^{-c}).
\end{align}

\begin{claim}\label{cl:21}
For any graph $H$ and any positive integers $k$ and $t$ we have 
$$
\Pb(X_{H,k-1} = t) \geq \frac{k-2t}{k} \cdot \Pb(X_{H,k} = t).
$$
\end{claim}

\begin{proof}
Sample uniformly at random a vertex set $A_{k-1} \subseteq V(H)$ of size $k-1$, using the following two steps. First sample $k$ vertices uniformly at random and without replacement, obtaining a set $A_k \subseteq V(H)$. Then, choose a vertex $u \in A_k$ uniformly at random and put $A_{k-1} = A_k \setminus \{u\}$. Observe that, if in the first step we sampled a $t$-edge graph, then the probability to `destroy' it in the second step is at most $2t/k$. Therefore
$$
\Pb(X_{H,k-1} = t) \geq \Pb(e(G[A_{k-1}]) = t | e(G[A_k]) = t) \cdot \Pb(e(G[A_k]) = t) \geq \frac{k-2t}{k} \cdot \Pb(X_{H,k} = t).
$$
\end{proof}

\begin{claim}\label{cl:31}
For any graph $H$ and any positive integers $k$ and $t$ we have 
$$
\Pb(X_{H,k-1} = t) \leq \Pb(X_{H,k} = t) + \frac{2t+2}{k}.
$$
\end{claim}

\begin{proof}
Consider the same two-step sampling as in the proof of Claim~\ref{cl:21}. Note that, if in the first step, the set $A_k$ contains more than $t$ edges, then the probability to obtain a $t$-edge graph in the second step is at most $(2t+2)/k$ (this is because, for every $s \geq 1$, a graph with $t+s$ edges contains at most $2t+2$ vertices of degree $s$). Thus, 
\begin{align*}
\Pb(X_{H,k-1}=t) &= \sum_{i=0}^{\binom{k}{2}} \Pb(X_{H,k} = i) \cdot \Pb(X_{H,k-1} = t | e(G[A_k]) = i)\\
&\leq \Pb(X_{H,k} = t) + \sum_{i=t+1}^{\binom{k}{2}} \Pb(X_{H,k} = i) \cdot \Pb(X_{H,k-1} = t | e(G[A_k]) = i)\\
&\leq \Pb(X_{H,k} = t) + \frac{2t+2}{k} \sum_{i=t+1}^{\binom{k}{2}} \Pb(X_{H,k}=i)\\
&\leq \Pb(X_{H,k} = t) + \frac{2t+2}{k}.
\end{align*}
\end{proof}

Let $b$ denote $\Pb(X_{G[R],k-2}=1)$. Then, applying Claim~\ref{cl:21} to~\eqref{eq:main2}, we obtain
\begin{align*} 
a &\leq (1 + o_k(1)) \cdot (\Pb(X_{G[R],k-1} = 1) \cdot e^{-c} + \Pb(X_{G[R],k-2} = 0) \cdot c e^{-c})\\
&\leq (1 + o_k(1)) \cdot (\Pb(X_{G[R],k-2} = 1) \cdot e^{-c} + \Pb(X_{G[R],k-2} = 0) \cdot c e^{-c})\\
&\leq (1 + o_k(1)) \cdot (b e^{-c} + (1-b) c e^{-c}).
\end{align*}
Observe that, if $c \geq 1$, then
$$
a \leq (1+o_k(1)) \cdot (b e^{-c} + (1-b) c e^{-c}) \leq (1+o_k(1)) \cdot (b c e^{-c} + (1-b) c e^{-c}) \leq 1/e + o_k(1).   
$$ 

Assume then that $c \leq 1$. Observe that, in this case, $b e^{-c} + (1-b) c e^{-c}$ is an increasing function in $b$. Moreover, 
$$
b \leq (1+o_k(1)) \Pb(X_{G[R],k} = 1) \leq (1+o_k(1)) a,
$$
where the first inequality holds by Claim~\ref{cl:31} (applied twice), and the second inequality holds since $a = \ind(k,1)$ and $R$ is large. 
Hence, 
$$
(1+o_k(1)) a \leq b e^{-c} + (1-b) c e^{-c} \leq a e^{-c} + (1-a) c e^{-c}.
$$
This implies that
$$
a \leq (1 + o_k(1)) \frac{c e^{-c}}{1-e^{-c} + c e^{-c}} = (1+o_k(1)) \frac{c}{e^c-1+c} \leq 1/2 + o_k(1).
$$

\subsection{The $3/4$-bound for a fixed $\ell$}

We use a similar, but slightly more technical, argument to the one used in the case $\ell=1$. Let $\ell$ be a positive integer and let $k \gg \ell$. Let $a = \ind(k,\ell)$, and note that $a$ is bounded away from $0$ by a constant $\varepsilon(\ell)$. This can be seen, for instance, by considering an appropriate random graph. Our aim is to prove that $a \leq 3/4 + o_k(1)$.

Let $G$ be a graph on $n$ vertices, attaining the maximum density of induced $\ell$-edged graphs, that is, $\Pb(X_{G,k} = \ell) = a + o_n(1)$. 
Using analogous arguments to the ones used to prove Claims~\ref{cl:DeltaOnk} and~\ref{cl:DeltaOmegank}, we may assume that $\Delta(G) = cn/k$ for some constant $c > 0$. Let $v$ be a vertex of maximum degree in $G$, let $Q = N_G(v)$, and let $R = V(G) \setminus (Q \cup \{v\})$. By the Poisson-approximation of the hypergeometric distribution, for each fixed $t$, with probability of approximately $e^{-c} c^t/t!$, the set $A$ will contain exactly $t$ vertices from $Q$. It thus follows by Lemma~\ref{lem:zykov} that
\begin{align} \label{eq:main1}
& a + o_n(1) = \Pb(X_{G,k} = \ell | v \in A) \nonumber \\ 
&= \sum_{t=0}^{\ell} \Pb(X_{G,k} = \ell | v \in A \textrm{ and } |A \cap Q| = t) \cdot \Pb(|A \cap Q| = t | v \in A) \nonumber \\
&\leq (1+o_k(1)) \cdot \left(\Pb(X_{G[R],k-1} = \ell) \cdot e^{-c} + \sum_{t=1}^{\ell} \Pb(X_{G[R],k-1-t} \leq \ell-t) \cdot \frac{c^t}{t!} e^{-c} \right).
\end{align}

Let $b$ denote $\Pb(X_{G[R],k-\ell-1} = \ell)$. Then, applying Claim~\ref{cl:21} to~\eqref{eq:main1} a constant (depending only on $\ell$) number of times, we obtain
\begin{align*} 
a &\leq (1+o_k(1)) \cdot \left(\Pb(X_{G[R],k-1} = \ell) \cdot e^{-c} + \sum_{t=1}^{\ell} \Pb(X_{G[R],k-1-t} \leq \ell-t) \cdot \frac{c^t}{t!} e^{-c} \right)\\
&\leq (1+o_k(1)) \cdot \left(\Pb(X_{G[R],k-\ell-1} = \ell) \cdot e^{-c} + \sum_{t=1}^{\ell} \Pb(X_{G[R],k-\ell-1} \leq \ell-t) \cdot \frac{c^t}{t!} e^{-c} \right).\\
&\leq (1+o_k(1)) \cdot \left(b e^{-c} + (1-b) \sum_{t=1}^\ell \frac{c^t}{t!} e^{-c} \right).
\end{align*}
Assume first that $c < \log 2$. In this case we have 
$$
\sum_{t=1}^\ell \frac{c^t}{t!} < e^c - 1 < 2 - 1 = 1,
$$
which implies that $b e^{-c} + (1-b) \sum_{t=1}^{\ell} e^{-c}{c^t}/{t!}$ is an increasing function in $b$. Now observe that
$$
b \leq (1+o_k(1)) \Pb(X_{G[R],k} = \ell) \leq (1+o_k(1)) a,
$$
where the first inequality holds by Claim~\ref{cl:31} (applied $\ell + 1$ times), and the second inequality holds since $a = \ind(k,\ell)$ and $R$ is large. 
Hence, 
$$
(1+o_k(1)) a \leq b e^{-c} + (1-b) \sum_{t=1}^{\ell} \frac{c^t}{t!} e^{-c} \leq a e^{-c} + (1-a) \sum_{t=1}^{\ell} \frac{c^t}{t!} e^{-c}.
$$
This implies that
$$
(1+o_k(1)) a \leq \frac{e^{-c} \sum_{t=1}^{\ell} {c^t}/{t!}}{1 - e^{-c} + e^{-c} \sum_{t=1}^\ell {c^t}/{t!}} = \frac{\sum_{t=1}^{\ell} {c^t}/{t!}}{e^c - 1 + \sum_{t=1}^{\ell} {c^t}/{t!}} < \frac{\sum_{t=1}^{\ell} {c^t}/{t!}}{2 \sum_{t=1}^{\ell} {c^t}/{t!}} = \frac{1}{2} < \frac{3}{4}.
$$
It thus remains to consider the case $c \geq \log 2$. Recall that $d_G(v) = \Delta(G) = cn/k$ and let $w$ be a vertex of minimum degree in $G$. 
\begin{claim}\label{cl:mindegreesmall}
$d_G(w) = \delta(G) = o_k(n/k)$. 
\end{claim}

\begin{proof}
Suppose that $\delta(G)\geq\alpha n/k$ for some constant $\alpha = \alpha(\ell) > 0$. Let $G'$ be a subgraph of $G$ obtained by selecting each edge of $G$ independently at random with probability $\beta/k$, where $\beta = \beta(\ell, \alpha)$ is a sufficiently large constant. Since $n \gg k$, with high probability all degrees in $G'$ will be between $(1-o(1)) \alpha \beta n/k^2$ and $(1+o(1)) c \beta n/k^2$, and by compactness we may assume that $e(G) = (1+o_k(1)) \mu n^2/k^2$ for some constant $\mu = \mu(\ell, \alpha, \beta) > \alpha \beta/3 > 0$. Hence $G'$ satisfies the conditions of Proposition~\ref{prop:Brun}, implying that the distribution of $X_{G',k}$ is approximately Poisson with parameter $\mu$. However, since $\beta$ was chosen to be sufficiently large, this means that $A_{G',k}$, and a fortiori $A_{G,k}$ will induce more than $\ell$ edges with probability $1 - o_k(1)$.
\end{proof}

Now, let $B = A_{G \setminus \{v,w\}, k-1}$ denote a subset of $V(G) \setminus \{v,w\}$ of size $k-1$, chosen uniformly at random among all such subsets. 
Assuming that $a > 1/2$ (as otherwise we are done), we apply Lemma~\ref{lem:sameDegree} to $v$ and $w$ to infer
$$
\Pb\left(e_G(v, B) = e_G(w, B)\right) > 2a - 1 - o_n(1).
$$
However, since $d_G(w) = o_k(n/k)$ holds by Claim~\ref{cl:mindegreesmall} and $d_G(v) \geq \log 2 (n/k)$ holds by assumption, a straightforward calculation shows that 
$$
\Pb(e_G(w,B) = 0) = 1 - o_k(1) \ \ \ \text{and} \ \ \ \Pb(e_G(v,B) = 0) \leq (1+o_k(1)) e^{-\log 2} = 1/2 + o_k(1).
$$ 
It follows that $1/2 + o_k(1) > 2a - 1 - o_n(1)$, implying $a \leq 3/4 + o_k(1)$, as claimed. 

\section{Concluding remarks} \label{sec:remarks}

\subsection*{Bounds on $\varepsilon$ in Theorem~\ref{thm:main}.}
While, for clarity of presentation, we did not make an effort to calculate $\varepsilon$ explicitly, it is not difficult to see that this constant is not too small. when $k$ is sufficiently large, the value $\varepsilon = 1/100$ is certainly sufficient, and with some care, one can improve it to $\varepsilon = 1/10$. On the other hand, looking at some of our arguments, it is evident that in order to go below $1/2$ for every sufficiently large $k$ and every $0 < \ell < \binom{k}{2}$, one would need new ideas.  

\subsection*{Upper bounds for fixed $\ell$}
It is not hard to see that our argument in the proof of Theorem~\ref{thm:fixed} gives in fact better bounds than $1/2 + o(1)$ for $\ell = 1$ and $3/4 + o(1)$ in the general case. Take for example the case $\ell = 1$. For $\Delta(G) = cn/k$ we obtain 
$$
a \leq (1+o(1)) \frac{c}{e^c-1+c},
$$  
where the right hand side attains its maximum of $1/2$ when $c$ is close to zero. On the other hand, when $\Delta(G) = o(n/k)$ Proposition~\ref{prop:Brun} (Brun's sieve) implies that $a \leq 1/e + o(1)$. Interpolating between these two arguments shows that indeed 
$$
\ind(k,1) < 1/2 - \Omega(1).
$$ 

\subsection*{Strengthening Conjecture~\ref{conj:stat} in some ranges}
Theorem~\ref{thm:poly} demonstrates that the assertion of Conjecture~\ref{conj:stat} holds, with room to spare, for almost all values of $k$ and $\ell$. 
In particular, $\ind(k, \ell)$ tends to $0$ with $k$ as long as $\ell$ and $\binom{k}{2}-\ell$ are quadratic in $k$.
We believe that in the above statement `quadratic' can be replaced with `super-linear'. 
\begin{conj} \label{conj::largek}
For all pairs $(k, \ell)$ satisfying $\min \left\{\ell, \binom{k}{2} - \ell \right\} = \omega(k)$, we have $\ind(k,\ell) = o(1)$.
\end{conj}
On the other hand, our construction for $\ell = k-1$ can be straightforwardly extended to show that for any fixed integer $C$ we have $\ind(k, C(k-C)) = \Omega(1)$. 

Note also that the argument appearing in the proof of Theorem~\ref{thm:poly} can be applied in a wider range, namely, it can be used to prove that $\ind(k,\ell)$ is polynomially small in $k$ whenever $\min \left\{\ell, \binom{k}{2} - \ell \right\} = \Omega(k^{2 - \delta})$ for some explicit constant $\delta > 0$.

It would be interesting to determine the `correct' power of $k$ in various ranges. For instance, when $\ell$ and $\binom{k}{2} - \ell$ are quadratic in $k$, Theorem~\ref{thm:poly} yields $\ind(k, \ell) = O(k^{-0.1})$, whereas, for some values of $\ell = \Theta \left(k^2 \right)$, we have $\ind(k, \ell) = \Omega(k^{-1/2})$ --- take for instance $\ell = k^2/4$ and $G = K_{n/2,n/2}$. We believe that the latter bound is tight.
\begin{conj} \label{conj::quadraticrange}
For all pairs $(k, \ell)$ satisfying $\min \left\{\ell, \binom{k}{2} - \ell \right\} = \Omega(k^2)$, we have $\ind(k,\ell) = O(k^{-1/2})$.
\end{conj}
\subsection*{Minimum edge-inducibility}
A natural counterpart to Conjecture~\ref{conj:stat} would be to determine the asymptotic value of
$$
\eta(k) := \min \left\{\ind(k,\ell): 0 \leq \ell \leq \binom{k}{2} \right\},
$$
as $k$ tends to $\infty$. Note that this question has been well-studied in the setting of graph-inducibilities~\cite{HT, KNV, PG, Yuster}. In particular, it is known~\cite{PG, Yuster} that 
$$
\min \{\ind(H) : |H|=k\} = (1+o(1)) \frac{k!}{k^k} = e^{-k+o(1)}.
$$ 
This is in stark contrast with the lower bound of $\eta(k)=\Omega(1/k)$ for edge-inducibilities, achieved by the random graphs $G \left(n, \ell/\binom{k}{2} \right)$. 

\subsection*{$3/4$ as a general upper bound}
As noted in the introduction, $\ind(k,\ell) < 1$ for every positive integer $k$ and every $0 < \ell < \binom{k}{2}$. It seems plausible that in fact $\ind(k,\ell) \leq 3/4$ for all such pairs $(k,\ell)$. Note that, if true, this bound would be tight since, as noted in the introduction, $\ind(3,1) = \ind(3,2) = 3/4$. Such a result would simultaneously improve Theorems~\ref{thm:main} and~\ref{thm:fixed}.

\subsection*{Hypergraphs}
The concepts of both graph- and edge-inducibility extend naturally to $r$-uniform hypergraphs. Since the lower bound constructions of $1/e$ for $\ell=1$ and for the star $K_{1,k-1}$ extend to higher uniformities as well, it makes sense to ask if  Conjectures~\ref{conj:stat} and~\ref{conj:ind} would also hold in this more general setting. Needless to say that we expect these questions to be difficult.

\subsection*{Further questions}
The constant of $1/e$, being closely related to the Poisson distribution, makes an appearance in many combinatorial and probabilistic setups. In pariticular, we  would like to draw the reader's attention to the interesting conjectures of Feige~\cite{Feige} and Rudich (see~\cite{KSS}). 

\section*{Addendum}

In the period when this paper was under review, our results have been extended in several directions. In particular, The Edge-Statistics Conjecture (Conjecture~\ref{conj:stat}) was proved in the superlinear regime by Kwan, Sudakov and Tran~\cite{KST}. In the sublinear regime the conjecture was proved by Fox and Sauermann~\cite{FS} and, independently, by Martinsson, Mousset, Noever and Truji\'{c}~\cite{MMNT}. Thus, the results of~\cite{KST}, combined with either~\cite{FS} or~\cite{MMNT} prove the assertion of Conjecture~\ref{conj:stat}. Some questions which remain open can be found in these papers.

\section*{Acknowledgement}

We thank the anonymous referee for carefully reading our paper and for providing helpful remarks.


\begin{thebibliography}{}

\bibitem{AK} R.~Ahlswede and G.~O.~H.~Katona, Graphs with maximal number of adjacent pairs of edges, \emph{Acta Math. Acad. Sci. Hungar.} 32 (1978), 97-–120.

\bibitem{AS}
N.~Alon and J.~H.~Spencer, \textbf{The Probabilistic Method},
Wiley-Interscience Series in Discrete Mathematics and Optimization, John Wiley \& Sons, fourth edition (2016).

\bibitem{AGK} N.~Alon, G.~Gutin and M.~Krivelevich, Algorithms with large domination ratio, \emph{Journal of Algorithms} 50 (2004), 118--131.

\bibitem{Baletal} 
J.~Balogh, P.~Hu, B.~Lidick\'{y} and F.~Pfender, Maximum density of induced 5-cycle is achieved by an iterated blow-up of 5-cycle, \emph{European Journal of Combinatorics} 52 (2016), 47--58.

\bibitem{Byellow}
B.~Bollob\'{a}s, \textbf{Modern Graph Theory}, Springer, (1998).

\bibitem{BS} 
J.~I.~Brown and A.~Sidorenko, The inducibility of complete bipartite graphs, \emph{Journal of Graph Theory} 18(6) (1994), 629--645.

\bibitem{Erd} 
P.~Erd\H{o}s, On a lemma of Littlewood and Offord, \emph{Bulletin of the American Mathematical Society} 51(12) (1945), 898--902.

\bibitem{Feige}
U.~Feige, On sums of independent random variables with unbounded variance, and estimating the average degree in a graph, in: \emph{Proceedings
of the Thirty-Sixth Annual ACM Symposium on the Theory of Computing} (2004), 594-–603.

\bibitem{FS} J.~Fox and L.~Sauermann. A completion of the proof of the Edge-statistics Conjecture. arXiv:1809.01352

\bibitem{HT} D.~Hefetz and M.~Tyomkyn, On the inducibility of cycles, \emph{Journal of Combinatorial Theory, Series B} 133 (2018), 243--258.

\bibitem{KNV}
D.~Kr\'{a}l, S.~Norin and J.~Volec, A bound on the inducibility of cycles, https://arxiv.org/abs/1801.01556.

\bibitem{JLR} S.~Janson, T.~{\L}uczak and A.~Ruci\'{n}ski, \textbf{Random graphs}. John Wiley \& Sons (2000).

\bibitem{KSS}
J.~Kahn, M.~Saks and C.~Smyth, The dual BKR inequality and Rudich's conjecture, 
\emph{Combin. Probab. Comput} 20(2) (2011), 257--266. 

\bibitem{KST} M.~Kwan, B.~Sudakov and T.~Tran. Anticoncentration for subgraph statistics. \emph{Journal of the LMS}, to appear.  https://arxiv.org/abs/1807.05202

\bibitem{MMNT} A.~Martinsson, F.~Mousset, A.~Noever and M.~Truji\'{c}. The edge-statistics conjecture for $\ell\ll k^{6/5}$. Preprint. arXiv:1809.02576.

\bibitem{PG} 
N.~Pippenger and M.~C.~Golumbic, Inducibility of graphs, \emph{Journal of Combinatorial Theory, Series B} 19(3) (1975), 189--203.  

\bibitem{Yuster}
R.~Yuster, On the exact maximum induced density of almost all graphs and their inducibility, \emph{preprint} https://arxiv.org/abs/1801.01047.

\bibitem{Zy} 
A.~A.~Zykov, On some properties of linear complexes, \emph{Mat. Sbornik N.S.} 24(66) (1949), 163--188.

\end{thebibliography}
\end{document}